\documentclass{amsart}
\usepackage{amssymb,xypic}
\usepackage[french,english]{babel}

\renewcommand{\hat}{\widehat}
\renewcommand{\tilde}{\widetilde}

\newcommand{\op}{{\operatorname{op}}}
\newcommand{\alg}{{\operatorname{alg}}}
\newcommand{\Id}{\operatorname{Id}}
\newcommand{\Z}{\mathbb{Z}}
\newcommand{\half}{{\textstyle\frac12}}

\newcommand{\qf}[1]{{\langle{#1}\rangle}} % formes quadratiques
\newcommand{\ba}{\overline{\rule{2.5mm}{0mm}\rule{0mm}{4pt}}} %canonical involution

\newcommand{\sfA}{\mathsf{A}}
\newcommand{\sfD}{\mathsf{D}}
\newcommand{\sfE}{\mathsf{E}}

\DeclareMathOperator{\End}{End}
\DeclareMathOperator{\ad}{ad}
\DeclareMathOperator{\gr}{\mathsf{gr}}
\DeclareMathOperator{\Int}{Int}
\DeclareMathOperator{\Aut}{Aut}

\DeclareMathOperator{\Cor}{Cor}
\DeclareMathOperator{\charac}{char}
\DeclareMathOperator{\Nrd}{Nrd}
\DeclareMathOperator{\nrd}{Nrd}
\DeclareMathOperator{\Trd}{Trd}
\DeclareMathOperator{\disc}{disc}
\DeclareMathOperator{\Sym}{Sym}

\DeclareMathOperator{\GO}{GO}
\DeclareMathOperator{\GSp}{GSp}
\DeclareMathOperator{\PGO}{PGO}
\DeclareMathOperator{\PGU}{PGU}
\DeclareMathOperator{\Sim}{Sim}

\DeclareMathOperator{\gPGO}{\mathbf{PGO}}
\DeclareMathOperator{\gPGU}{\mathbf{PGU}}
\DeclareMathOperator{\gSU}{\mathbf{SU}}

\DeclareMathOperator{\gPGL}{\mathbf{PGL}}
\DeclareMathOperator{\gSpin}{\mathbf{Spin}}

\DeclareMathOperator{\gAut}{\mathbf{Aut}}

\newtheorem{prop}{Proposition}[section]
\newtheorem{cor}[prop]{Corollary}
\newtheorem{thm}[prop]{Theorem}
\newtheorem{lemma}[prop]{Lemma}
\theoremstyle{definition}
\newtheorem{example}[prop]{Example}
\theoremstyle{remark}

\newtheorem{remarkk}[prop]{Remark}

\title[Outer automorphisms of classical algebraic groups]{Outer
  automorphisms of classical algebraic groups}  
\author{Anne Qu\'eguiner-Mathieu}
\address{Universit\'e Paris 13\\
Sorbonne Paris Cit\'e\\
LAGA - CNRS (UMR 7539)\\
F-93430 Villetaneuse, France}
\email{queguin@math.univ-paris13.fr}
\thanks{The first author acknowledges the support of the French Agence Nationale de la Recherche (ANR) under reference ANR-12-BL01-0005.}

\author{Jean-Pierre Tignol}
\address{ICTEAM Institute, Box L4.05.01\\
Universit\'e catholique de Louvain\\
B-1348 Louvain-la-Neuve, Belgium}
\email{jean-pierre.tignol@uclouvain.be}
\thanks{The second author is grateful to the first author
  and the Universit\'e Paris 13 for their hospitality while the work
  for this paper was carried out. He acknowledges support
    from the Fonds de la Recherche Scientifique--FNRS under grant
    n$^\circ$~J.0014.15.}


\date{October 15, 2016}

\begin{document}
\maketitle
\begin{abstract}
  The so-called Tits class, associated to an adjoint absolutely almost
  simple algebraic group, provides a cohomological obstruction for
  this group to admit an outer automorphism. If the group has inner
  type, this obstruction is the only one. In this paper, we prove this
  is not the case for classical groups of outer type, except
  for groups of type $^2\sfA_n$ with 
  $n$ even, or $n=5$.  More precisely, we prove a descent theorem for
  exponent $2$ and degree $6$ algebras with unitary involution, which
  shows that their automorphism groups have outer automorphisms. In
  all other relevant classical types, namely $^2\sfA_n$ with $n$ odd,
  $n\geq3$ and $^2\sfD_n$, we provide explicit examples where the Tits
  class obstruction vanishes, and yet the group does not have
  outer automorphism. As a crucial tool, we use ``generic''
    sums of algebras with involution.
\end{abstract}

\selectlanguage{french}
\begin{abstract}
  \`A tout groupe alg\'ebrique absolument presque simple de type adjoint
  est associ\'ee une classe de cohomologie connue sous le nom de
  \flqq classe de Tits\frqq\, qui donne une obstruction \`a l'existence
  d'automorphismes ext\'erieurs. Si le groupe est de type int\'erieur, il
  n'y a pas d'autre obstruction. Dans ce travail, nous montrons qu'il
  n'en va pas de m\^eme pour les groupes classiques de type ext\'erieur,
  sauf pour les groupes de type $^2\sfA_n$ avec $n$ pair ou
  $n=5$. Plus pr\'ecis\'ement, nous \'etablissons 
  pour les alg\`ebres \`a involution unitaire de degr\'e~$6$ et
  d'exposant~$2$ un th\'eor\`eme de descente qui montre que les groupes
  d'automorphismes de ces alg\`ebres ont des automorphismes
  ext\'erieurs. Pour les types $^2\sfA_n$ avec $n$ impair, $n\geq3$, et
  les types $^2\sfD_n$, nous construisons des exemples explicites o\`u
  l'obstruction donn\'ee par la classe de Tits est nulle alors que le
  groupe ne poss\`ede pas d'automorphisme ext\'erieur. Un outil crucial de
  nos constructions
  est la somme \flqq g\'en\'erique\frqq\ d'alg\`ebres \`a involution.
\end{abstract}
\selectlanguage{english}

\section{Introduction}  

Every automorphism of an absolutely almost simple algebraic group
scheme $G$ of adjoint type over an arbitrary field $F$ induces an
automorphism of its Dynkin diagram $\Delta$. Inner automorphisms of
$G$ act trivially on $\Delta$, and there is an exact sequence of
algebraic group schemes
\begin{equation}
\label{eq:exact}
1\to G \to \gAut(G) \to \gAut(\Delta) \to 1,
\end{equation}
see \cite[Exp.~XXIV, 1.3, 3.6]{SGA}. If $G$ is split, the
corresponding sequence of groups of rational points is exact and
split, see~\cite[(25.16)]{KMRT}, \cite[\S16.3]{Sp}. Therefore, a split
adjoint group $G$ 
admits outer automorphisms if and only if its Dynkin diagram admits
automorphisms, i.e., if $G$ has type $\sfA_n$ with $n\geq 2$, $\sfD_n$
with $n\geq 3$ or $\sfE_6$. Moreover, in all three cases,  
$\gAut(\Delta)(F)$ lifts to an isomorphic subgroup in
$\gAut(G)(F)$. This property does not hold generally for nonsplit
groups. For instance, if $G$ is the connected component of the
identity in the group scheme of automorphisms of a central simple
$F$-algebra with quadratic pair $(A,\sigma,f)$, then $G$ has no outer
automorphisms if $A$ is not split by the quadratic \'etale $F$-algebra
defined by the discriminant of the quadratic pair, see
\S\,\ref{translation-orthogonal.section} below. More generally,
Garibaldi identified in \cite[\S2]{Gar} a 
cohomological obstruction to the existence of outer automorphisms of
an arbitrary absolutely almost simple algebraic group scheme $G$: the
group $\gAut(\Delta)(F)$ acts on $H^2(F,C)$, where $C$ is the center
of the simply connected group scheme isogenous to $G$, and the Tits
class $t_G\in H^2(F,C)$ is invariant under the action of the image of
$\gAut(G)(F)$ in $\gAut(\Delta)(F)$. Therefore, automorphisms of
$\Delta$ that do not leave $t_G$ invariant do not lift to outer
automorphisms of $G$. For adjoint or simply connected groups of inner
type, Garibaldi showed in 
\cite[\S2]{Gar} that this is the only obstruction to the lifting of
automorphisms of $\Delta$. As he explains in~\cite[Thm 11]{Gar} this
has interesting consequences in Galois cohomology. In a subsequent
paper, Garibaldi--Petersson \cite[Conjecture~1.1.2]{GP} 
conjectured that this Tits class obstruction is the only obstruction,
also for adjoint or simply connected groups of outer type.

In this paper, we provide a complete answer to the question raised by
Garibaldi and Petersson for groups of outer type $\sfA$ and $\sfD$,
leaving aside trialitarian groups (see the Appendix). Thus, in all the
cases we consider, $\gAut(\Delta)(F)$ has order~$2$. Our main goal is
to compare the following 
three conditions, listed from weaker to stronger: 
\begin{description}
\item[(Out~1)] The Tits class $t_G$ is fixed under $\gAut(\Delta)(F)$;
\item[(Out~2)] $G$ admits an outer automorphism defined over $F$;
\item[(Out~3)] $G$ admits an outer automorphism of order $2$ defined
  over $F$.  
\end{description}
Under condition (Out~2), the sequence
\[
1\to G(F)\to \gAut(G)(F)\to
\gAut(\Delta)(F)\to 1
\]
is exact, and under condition (Out~3), it is split. 
In~\cite{Gar}, Garibaldi proves that all three conditions are
equivalent if $G$ has inner type $\sfA$ or $\sfD$ (see
Remarks~\ref{innerA.rem} and \ref{orthogonal.rem}).  
This is not the case for groups of
outer type, and our main result is the following: 
 
\begin{thm}
  \label{main.thm}
  Let $G$ be an absolutely almost simple adjoint or simply connected
  algebraic group 
  scheme of type $^2\sfA_n$, with $n\geq 2$, or $^2\sfD_n$, with
  $n\geq 3$.
  \begin{enumerate}
  \item[(1)] If $G$ has type $^2\sfA_n$, with $n$ even, or $^2\sfA_5$, then
    conditions (Out~1), (Out~2) and (Out~3) are equivalent.
  \item[(2)] In all the other types, there are examples of groups for which
    (Out~1) holds and (Out~2) does not hold, and examples of groups
    for which (Out~2) holds and (Out~3) does not hold.
  \end{enumerate}
\end{thm}

In other words, assertion (2) says there are examples where the
condition on the Tits class is satisfied, and yet $G$ does not have
any outer automorphism, and examples where $G$ has an outer
automorphism, but no outer automorphism of order $2$. In particular,
this disproves Conjecture~1.1.2 in~\cite{GP}, and provides
  examples of simply connected absolutely simple algebraic group
  schemes $G$ for which the 
Galois cohomology sequence 
\[
H^1(F,C)\to H^1(F,G)\to H^1(F,\gAut(G))
\]
from~\cite[Thm 11(b)]{Gar} (where $C$ is the center of $G$) is not exact.

\medbreak

Every absolutely almost simple algebraic group scheme of adjoint type
$^2\sfA_n$ over $F$ is isomorphic to $\gPGU(B,\tau)=\gAut_K(B,\tau)$
for some central simple algebra $B$ of degree $n+1$ over a separable
quadratic field extension $K$ of $F$ with a $K/F$-unitary involution
$\tau$. As explained below in~\S\,\ref{unitary.sec}, condition (Out~1)
holds for the group $\gPGU(B,\tau)$ if and only if $B$ has exponent at
most~$2$, and condition (Out~3) holds if and only if $(B,\tau)$ has a
descent, i.e., $(B,\tau)=(B_0,\tau_0)\otimes_F(K,\iota)$ for some
central simple $F$-algebra with $F$-linear involution $(B_0,\tau_0)$.
For $n$ even, Theorem~\ref{main.thm}(1) can be reformulated in a more
precise form:

\begin{thm}
  \label{minor.thm}
  Let $(B,\tau)$ be a central simple algebra with unitary
  involution. If $\deg B$ is odd, then conditions (Out~1), (Out~2),
  and (Out~3) for $\gPGU(B,\tau)$ are equivalent and hold if and only
  if $B$ is split.
\end{thm}

The proof is easy: see Corollary~\ref{GPCor912.cor}.
\smallbreak

Now, assume $G=\gPGU(B,\tau)$ has type $^2\sfA_5$, i.e., $B$ has
degree $6$. If the 
exponent of $B$ is at most~$2$, then its index is at most
$2$. Therefore, Theorem~\ref{main.thm}(1) for such groups follows from
the following descent theorem for algebras with unitary involution,
proved in \S\ref{subsec:unitsim}: 

\begin{thm}
\label{descent.thm}
Let $(B,\tau)$ be a central simple algebra of degree at most~$6$ and
index at most~$2$, with a $K/F$-unitary involution. There exists a central
simple algebra with orthogonal involution $(B_0,\tau_0)$ over $F$, of
the same index as $B$, such
that $(B,\tau)=(B_0,\tau_0)\otimes(K,\iota)$, where $\iota$ is the
unique nontrivial $F$-automorphism of $K$.  
\end{thm} 

It also follows from this theorem that assertion~(1) does hold for
groups of type $^2\sfA_3$ when the underlying algebra $B$ has index
at most~$2$; but this does not apply to all groups of type $^2\sfA_3$,
since a degree $4$ central simple algebra of exponent $2$ can be of
index $4$.  
An example of a degree $4$ and exponent $2$ algebra with unitary
involution that does not have a descent will be provided in
\S\,\ref{Doddexample.section} below (see Remark~\ref{rem:A3D3}).  

As usual for classical groups, we use as a crucial tool their explicit
description in terms of algebras with involution or quadratic pair.
How conditions (Out~1), (Out~2) and (Out~3) translate into conditions on
these algebraic structures is explained in
\S\,\ref{translation.section}. Section~\ref{orthogonal.section}
studies in more details the $^2\sfD_n$ case. In particular, we
introduce our main tool for proving assertion~(2) of
Theorem~\ref{main.thm}, namely ``generic'' orthogonal sums of hermitian
forms or involutions. In \S\,\ref{sec:outunit}, using the same kind of
strategy, we prove Theorem~\ref{descent.thm}, and complete the proof of
Theorem~\ref{main.thm} by producing examples of outer type $^2\sfA_n$.

We refer the reader to~\cite{KMRT} for definitions and basic facts on
central simple algebras, involutions, and quadratic pairs.  Recall
that if $\charac F\not =2$, then for any quadratic pair $(\sigma,f)$,
$\sigma$ is an involution of orthogonal type, and $f$ is the map
defined on the set $\Sym(A,\sigma)$ of $\sigma$-symmetric elements by
$f(s)=\frac 12 \Trd_A(s)$. Hence the quadratic pair is uniquely
determined by the involution, and we usually write $(A,\sigma)$ for
$(A,\sigma,f)$ in this case.  

\subsection*{Notation}
If $\mathfrak A$ is a structure (such as an algebra with involution or
an algebraic group scheme) defined over a field $F$, we write
$\gAut(\mathfrak{A})$ for the algebraic group scheme of automorphisms
of $\mathfrak A$ and $\Aut(\mathfrak{A})$ for its (abstract) group of
rational points:
\[
\Aut(\mathfrak{A})=\gAut(\mathfrak{A})(F).
\]
We use a similar convention for classical groups; thus for instance if
$(B,\tau)$ is a central simple algebra with unitary involution over a
separable quadratic field extension $K$ of $F$, then 
\[
\gPGU(B,\tau)=\gAut_K(B,\tau) \qquad\text{and}\qquad
\PGU(B,\tau)=\gPGU(B,\tau)(F).
\] 

Note that an absolutely almost simple simply connected algebraic group
scheme and its isogenous adjoint group have the same automorphism
group, hence it is enough to consider adjoint groups.  For isogenous
groups that are neither adjoint nor simply connected, obstruction to
the existence of an outer automorphism can arise from the fundamental
group.

\section{Groups of type $\sfA$ and $\sfD$, and associated algebras
  with involution} 
\label{translation.section}

The main purpose of this section is to point out how conditions (Out~1),
(Out~2) and (Out~3) can be translated in terms of the corresponding
algebra with involution or quadratic pair. Part of
Theorem~\ref{main.thm} follows immediately, as we will
show. Throughout this section, $F$ is an arbitrary field.  

\subsection{Type $\sfA$}
\label{unitary.sec}

Let $K$ be an \'etale quadratic $F$-algebra, and $\iota$ be the
nontrivial $F$-automorphism of $K$. Consider a central simple
$K$-algebra with $K/F$-unitary involution $(B,\tau)$. We denote by
$({}^\iota B,{}^\iota\tau)$ the conjugate algebra with involution
defined by $^\iota B=\{{}^\iota x\mid x\in B\}$ with the operations
\[
{}^\iota x +{}^\iota y={}^\iota(x+y), \quad {}^\iota x\,{}^\iota
y={}^\iota(xy),\quad \lambda\,{}^\iota x={}^\iota(\iota(\lambda) x)
\quad\text{and}\quad
{}^\iota \tau({}^\iota x)={}^\iota(\tau(x))
\]
for $x$, $y\in B$ and $\lambda\in K$.

The following propositions were proven by Garibaldi--Petersson
\cite{GP}: 

\begin{prop}
  \label{unitary.prop}
  Let $G=\gPGU(B,\tau)$, with $\deg B\geq 3$.
  \begin{enumerate}
  \item Condition (Out~1) holds for $G$ if and only if $B$ has
    exponent at most $2$;
  \item Condition (Out~2) holds for $G$ if and only if $(B,\tau)$
    admits a $\iota$-semilinear automorphism, i.e., $(B,\tau)$ is
    isomorphic to $(^\iota B,{}^\iota \tau)$;
  \item Condition (Out~3) holds for $G$ if and only if $(B,\tau)$
    admits a $\iota$-semilinear automorphism of order $2$.
  \end{enumerate}
\end{prop}

\begin{prop}
  \label{descent1.prop}
  Condition (Out~3) holds for $\gPGU(B,\tau)$ if and only if
  $(B,\tau)$ has a descent, i.e., there exists a central simple
  $F$-algebra with $F$-linear involution $(B_0,\tau_0)$ such that
  $(B,\tau)\simeq (B_0,\tau_0)\otimes (K,\iota)$.
\end{prop}

\begin{proof}[Proof of Proposition~\ref{unitary.prop}]
  Those assertions are taken from~\cite[\S\,9]{GP}; for the reader's
  convenience, we briefly sketch an argument.  One may understand the
  action of $\gAut(\Delta)(F)$ on the Tits class by looking at the
  action on the Tits algebras. For groups of type $\sfA$, the symmetry
  of the diagram, together with the description of the Tits algebras
  given in~\cite[\S\,27.B]{KMRT}, shows that $t_{\gPGU(B,\tau)}$ is
  invariant under the action of $\gAut(\Delta)(F)$ if and only if $B$ is
  invariant under the action of the Galois group of $K/F$, i.e., if
  $B$ is isomorphic to its conjugate $^\iota B$. Since $\tau$ is a
  semilinear involution, it induces an anti-automorphism between $B$
  and $^\iota B$. Therefore, $B$ and $^\iota B$ are isomorphic if and
  only if $B$ is isomorphic to its opposite algebra, i.e., $B$ has
  exponent at most $2$.

  Recall from \cite[(26.9)]{KMRT} that there is an equivalence of
  categories between the groupoid $\sfA_n(F)$ of central simple
  algebras of degree~$n+1$ with a unitary involution over some \'etale
  quadratic $F$-algebra and the groupoid $\overline{\sfA}^n(F)$ of
  adjoint absolutely almost simple linear algebraic groups of type 
  $\sfA_n$ defined over $F$, under which $(B,\tau)$ maps to the adjoint
  group $\gPGU(B,\tau)$.  Hence, $\gPGU(B,\tau)$ and
  $(B,\tau)$ have the same automorphisms. More precisely, the
  automorphisms of $\gPGU(B,\tau)$ defined over $F$ coincide with the
  $F$-automorphisms of $(B,\tau)$, see~\cite[(26.10)]{KMRT}. Among
  those, the inner automorphisms are the $K$-linear automorphisms of
  $(B,\tau)$, while outer automorphisms coincide with
  $\iota$-semilinear automorphisms of $(B,\tau)$. Therefore,
  $\gPGU(B,\tau)$ admits an outer automorphism if and only if
  $(B,\tau)$ is isomorphic to $(^\iota B, {}^\iota \tau)$. Note that the
  condition $\deg B\geq 3$ is crucial here. Indeed, if $B=Q$ is a
  quaternion algebra, $\gPGU(Q,\tau)$ has no outer automorphism, while
  $(Q,\tau)$ does admit semilinear automorphisms. 
\end{proof}

\begin{proof}[Proof of Proposition~\ref{descent1.prop}]
  If $(B,\tau)\simeq (B_0,\tau_0)\otimes (K,\iota)$, then
  $\Id_{B_0}\otimes \iota$ is a semilinear automorphism of $B$ which
  commutes with $\tau=\tau_0\otimes \iota$, and has order
  $2$. Therefore, it induces an outer automorphism of $\gPGU(B,\tau)$
  of order $2$. Conversely, assume $(B,\tau)$ has a $\iota$-semilinear
  automorphism $\varphi$ of order~$2$. The $F$-algebra of fixed points
  $B_0=B^\varphi$ is a central simple $F$-algebra of the same degree
  as $B$, hence 
  \[
  B=B_0\otimes_FK.
  \]
  Moreover, since $\varphi$
  commutes with $\tau$, the restriction of $\tau$ induces an
  $F$-linear involution $\tau_0$ of $B_0$, and we have
  $(B,\tau)=(B_0,\tau_0)\otimes_F(K,\iota)$ as required.
\end{proof}

\begin{remarkk}
\label{innerA.rem}
If $G$ has inner type $^1\sfA_n$, then $K\simeq F\times F$ and the
corresponding algebra with involution $(B,\tau)$ is isomorphic 
to $(E\times E^\op, \varepsilon)$ for some central simple
$F$-algebra $E$, with $\varepsilon$ 
the exchange involution (see~\cite[(2.14)]{KMRT}). If condition
(Out~1) holds, then $E$ has exponent at most~$2$, hence by a theorem
of Albert (see \cite[(3.1)]{KMRT}) $E$ carries an $F$-linear
involution $\gamma$. Identifying $E\otimes_F (F\times F)$ with
$E\times E$, one may check that the map $(x,y)\in E\times E \mapsto
(x,\gamma(y)^\op)\in E\times E^\op$ induces an
isomorphism between $(E,\gamma)\otimes_F (F\times F, \iota)$ and
$(E\times E^\op, \varepsilon)$. Therefore $(E\times
E^\op, \varepsilon)$ has a descent, provided $E$ has
exponent at most $2$. This shows that conditions (Out~1), (Out~2) and
(Out~3) are equivalent for groups of inner type $^1\sfA_n$, as
observed by Garibaldi~\cite[Ex.~17(i)]{Gar}. Moreover,
these conditions hold if and only if $G=\gPGU(E\times E^\op,
\varepsilon)=\gPGL(E)$ with $E$ of exponent at most $2$. If $n$ is
even, then $E$ has odd degree $n+1$, and the conditions hold if and
only if $E$ is split. 
\end{remarkk} 

Combining Proposition~\ref{unitary.prop} and~\ref{descent1.prop} we
already get Theorem~\ref{minor.thm}. More precisely, we have  

\begin{cor}
  \label{GPCor912.cor}
  Let $G=\gPGU(B,\tau)$ with $\deg B\geq 3$.
  \begin{enumerate}
  \item If $B$ is split, then $G$ admits outer automorphims of order
    $2$.
  \item If $G$ has type $^2\sfA_n$, with $n$ even, conditions (Out~1),
    (Out~2) and (Out~3) are equivalent, and hold if and only if $B$ is
    split.
  \end{enumerate}
\end{cor}

\begin{proof}
  If $B$ is split, we may assume $B= \End_KV$ for some $K$-vector
  space $V$. Then $\tau$ is the adjoint
  involution with respect to some nondegenerate hermitian form
  $h\colon V\times V\rightarrow K$.  Pick a diagonalization of $h$,
  corresponding to a $K$-basis $(e_i)_{1\leq i\leq n}$ of $V$. For all
  $i$, we have $h(e_i,e_i)\in F^\times$, hence $h$ restricts to a
  nondegenerate 
  symmetric bilinear form $b$ on the $F$-vector space
  $V_0=e_1F+\dots +e_nF$. Therefore,
  $(B,\tau)=(\End_FV_0,\ad_b)\otimes_F(K,\iota)$ has a descent, so
  (Out~3) holds for $\gPGU(B,\tau)$.

  Now, assume that $G$ has type $^2\sfA_n$ for some $n\geq 3$, with $n$
  even. Then $G=\gPGU(B,\tau)$, where $B$ has odd degree $n+1$. Hence,
  under condition (Out~1), $B$ is split, so (Out~3) holds by the first
  assertion, and this concludes the proof.
\end{proof}

Corollary~\ref{GPCor912.cor} was proved by Garibaldi--Petersson,
see~\cite[Cor 9.1.2]{GP}. 
\smallbreak

To prove Theorem~\ref{main.thm}(2), we will give
in \S\,\ref{orthogonal.section} and \S\,\ref{sec:outunit} examples of
algebras with unitary involutions $(B,\tau)$ such that either $B$ has
exponent $2$ and $(B,\tau)$ is not isomorphic to its conjugate
$(^\iota B,{}^\iota \tau)$, or $(B,\tau)$ and $(^\iota B,{}^\iota \tau)$
are isomorphic, yet $(B,\tau)$ does not have a descent. We provide
examples of degree $4$ and index $4$, and examples of degree  
$n+1$ and index $2$ for all odd $n\geq 7$; see
Remark~\ref{rem:A3D3} and \S\,\ref{sec:outunitex}. 

\subsection{Type $\sfD$}
\label{translation-orthogonal.section}

Let $A$ be a central simple $F$-algebra of even degree, and let
$(\sigma,f)$ be a quadratic pair on $A$. We write $\GO(A,\sigma,f)$
for the (abstract) group of similitudes of $(A,\sigma,f)$, defined as
\[
\GO(A,\sigma,f)=\{g\in A^\times\mid \sigma(g)g\in F^\times \text{ and
} f\circ\Int(g)=f\}.
\]
The scalar $\mu(g)=\sigma(g)g$ is called the \emph{multiplier} of $g$. 
Mapping $g\in\GO(A,\sigma,f)$ to $\Int(g)$ yields an identification
of $\GO(A,\sigma,f)/F^\times$ with the group of rational points
$\PGO(A,\sigma,f)=\Aut(A,\sigma,f)$. Every automorphism of
$(A,\sigma,f)$ induces an automorphism of the Clifford algebra
$C(A,\sigma,f)$. A similitude is said to be \emph{proper} if the
induced automorphism of $C(A,\sigma,f)$ is the identity on the center $Z$;
otherwise it is said to be \emph{improper}.
The proper similitudes form a subgroup $\GO^+(A,\sigma,f)$ which
satisfies $\GO^+(A,\sigma,f)/F^\times=\PGO^+(A,\sigma,f)$ for
$\gPGO^+(A,\sigma,f)$ the connected component of the identity in
$\gPGO(A,\sigma,f)=\gAut(A,\sigma,f)$. 

If $A=\End_FV$ for some $F$-vector space $V$, then every quadratic
pair $(\sigma,f)$ on $A$ is adjoint to some nonsingular quadratic form
$q$ on $V$, see \cite[(5.11)]{KMRT}. In that case, we write
simply $\GO(V,q)$, $\gPGO(V,q)$, etc. for $\GO(A,\sigma,f)$,
$\gPGO(A,\sigma, f)$, etc.

\begin{prop}
  \label{orthogonal.prop}
  Let $G=\gPGO^+(A,\sigma,f)$, with $\deg A=2n\geq 4$,
  and let $Z$ be the discriminant quadratic $F$-algebra of
  $(\sigma,f)$, i.e., $Z$ is the center of the Clifford algebra
  $C(A,\sigma,f)$. Assume $Z$ is a field.
  \begin{enumerate}
  \item Condition (Out~1) holds for $G$ if and only if $A$ is split by
    $Z$;
  \item Condition (Out~2) holds for $G$ if and only if $(A,\sigma,f)$
    admits improper similitudes;
  \item Condition (Out~3) holds for $G$ if and only if $(A,\sigma,f)$
    admits square-central improper similitudes.
  \end{enumerate}
\end{prop}

In particular, condition (Out~1) holds if and only if the algebra $A$
is Brauer-equivalent to a quaternion algebra split by $Z$. This
condition is necessary for the existence of an improper similitude by
the generalization of Dieudonn\'e's theorem on multipliers of
similitudes given in~\cite[(13.38)]{KMRT}.

\begin{proof}
  (1): Let $\iota$ denote the nontrivial
  $F$-automorphism of $Z$, and let $C=C(A,\sigma,f)$. The Tits class
  $t_G$ is invariant under the 
  action of $\Aut(\Delta)$ if and only if $C$ is
  isomorphic to its conjugate algebra $^\iota C$, or equivalently
  $C\otimes_Z {\,}^\iota C^\op$ is split.  Recall from \cite[(9.12)]{KMRT}
  the fundamental relations between $A$ and $C$: if $n$ is even, then
  $C\otimes_ZC$ is split and the corestriction $\Cor_{Z/F}C$ is
  Brauer-equivalent to $A$. After scalar extension to $Z$, it follows
  from the latter relation that the $Z$-algebra $A_Z$ is
  Brauer-equivalent to $C\otimes_Z{}^\iota C$. If $n$ is odd, then
  $C\otimes_ZC$ is Brauer-equivalent to $A_Z$, while $\Cor_{Z/F}C$ is
  split, hence $C\otimes_Z{}^\iota C$ is split. Thus, in each case
  $A_Z$ is Brauer-equivalent to $C\otimes_Z{}^\iota C^\op$, and we get
  that (Out~1) holds for $G$ if and only if $A_Z$ is split.
\smallbreak

  (2) and (3): If $\deg A\neq8$, we may argue along the same lines as for
  Proposition~\ref{unitary.prop}, using the equivalence of categories
  between the 
  groupoid $\sfD_n(F)$ of central simple $F$-algebras of degree~$2n$
  with quadratic pair and the groupoid $\overline{\sfD}^n(F)$ of
  adjoint
  absolutely almost simple groups of type $\sfD_n$, which maps the
  algebra $A$ with quadratic pair $(\sigma,f)$ to
  $\gPGO^+(A,\sigma,f)$, see~\cite[(26.15)]{KMRT}. This line of
  argument does not apply to the case where $\deg A=8$, however,
  because the description of $\sfD_4(F)$ is different (see
  \cite{KMRT}). Therefore, we give a different proof, which applies in
  all cases where $\deg A=2n\geq4$.

  We will need the following lemma, which is probably well-known:
\renewcommand{\qed}{\relax}
\end{proof}

\begin{lemma}
  \label{PGOhyp.lem}
  Let $(V,q)$ be a hyperbolic space of dimension $2n\geq4$ over an
  infinite field $E$. The map
  $\PGO(V,q)\to\Aut\bigl(\gPGO^+(V,q)\bigr)$ which carries $gE^\times$
  to $\Int(g)$ is injective.
\end{lemma}

\begin{proof}
  Let $b$ be the polar bilinear form of $q$, and let $e_1$, $f_1$,
  \ldots, $e_n$, $f_n$ be a symplectic base of $(V,q)$, i.e., a base
  such that $q(e_i)=q(f_i)=0$ and
  \[
  b(e_i,f_i)=1,\quad b(e_i,e_j)=b(e_i,f_j)=b(f_i,f_j)=0 \text{ for
    all $i$, $j=1$, \ldots, $n$ with $i\neq j$.}
  \]
  Since $E$ is infinite, we may find $\alpha_1$, \ldots, $\alpha_n\in
  E^\times$ such that $\alpha_1$, $\alpha_1^{-1}$, \ldots, $\alpha_n$,
  $\alpha_n^{-1}$ are pairwise distinct and moreover, if $\charac
  E\neq2$,
  \[
  \{\alpha_1,\alpha_1^{-1}, \ldots, \alpha_n, \alpha_n^{-1}\} \neq
  \{-\alpha_1,-\alpha_1^{-1}, \ldots, -\alpha_n,-\alpha_n^{-1}\}.
  \]
  Consider the proper isometry $a\in\GO^+(V,q)$ defined by
  \[
  a(e_i)=\alpha_ie_i\qquad\text{and}\qquad a(f_i)=\alpha_i^{-1}f_i
  \quad\text{for $i=1$, \ldots, $n$.}
  \]
  Let $g\in\GO(V,q)$ be such that $\Int(g)$ is the identity on
  $\PGO^+(V,q)$. Then $g^{-1}ag=\lambda a$ for some $\lambda\in
  E^\times$. Because $g^{-1}ag$ and $a$ are isometries, we must have
  $\lambda=\pm1$. Moreover, by evaluating $ag=\lambda ga$ on $e_1$,
  \ldots, $f_n$, we obtain
  \[
  ag(e_i)=\lambda\alpha_ig(e_i)\qquad\text{and}\qquad
  ag(f_i)=\lambda\alpha_i^{-1} g(f_i) \qquad\text{for $i=1$, \ldots,
    $n$.}
  \]
  Thus, $g(e_i)$ (resp.\ $g(f_i)$) is an eigenvector of $a$ with
  eigenvalue $\lambda \alpha_i$ (resp.\ $\lambda\alpha_i^{-1}$). But
  the eigenvalues of $a$ are $\alpha_1$, $\alpha_1^{-1}$, \ldots,
  $\alpha_n$, $\alpha_n^{-1}$, hence
  \[
  \{\lambda\alpha_1,\lambda\alpha_1^{-1}, \ldots, \lambda\alpha_n,
  \lambda\alpha_n^{-1}\} =
  \{\alpha_1,\alpha_1^{-1}, \ldots, \alpha_n, \alpha_n^{-1}\}
  \]
  with $\lambda=\pm1$. By the choice of $\alpha_1$, \ldots, $\alpha_n$
  we must have $\lambda=1$, hence $g(e_i)$ must be a scalar multiple
  of $e_i$ and $g(f_i)$ a scalar multiple of $f_i$. Therefore, there
  exist $\gamma_1$, \ldots, $\gamma_n\in E^\times$ such that, letting
  $\mu=\mu(g)$ be the multiplier of $g$,
  \[
  g(e_i)=\gamma_ie_i\qquad\text{and}\qquad g(f_i)=\mu\gamma_i^{-1}f_i.
  \]
  Thus, the matrix of $g$ with respect to the base $e_1$, \ldots,
  $f_n$ is diagonal. Using \cite[(12.24)]{KMRT} if $\charac E\neq2$
  and \cite[(12.12)]{KMRT} if $\charac E=2$, it is then easy to check
  that $g$ is a proper similitude. Since the map of algebraic group
  schemes $\gPGO^+(V,q)\to\gAut(\gPGO^+(V,q))$ is injective
  (cf.~\eqref{eq:exact}), it follows that the homomorphism
  $\PGO^+(V,q)\to\Aut(\gPGO^+(V,q))$ is injective, hence $g\in E^\times$.
\end{proof}

\begin{proof}[Proof of Proposition~\ref{orthogonal.prop}(2) and (3)]
  The map $g\mapsto\Int(g)$ induces a map of algebraic group schemes
  $\Phi$ which fits in the following commutative diagram with exact
  rows:
  \[
  \xymatrix{
  1\ar[r]&\gPGO^+(A,\sigma,f) \ar[r] \ar@{=}[d] & \gPGO(A,\sigma, f)
  \ar[r] \ar[d]_{\Phi}&\gAut_F(Z)\ar[d]^{\Psi}\ar[r]&1\\
  1\ar[r]&\gPGO^+(A,\sigma,f)\ar[r]&\gAut(\gPGO^+(A,\sigma,f)) \ar[r]&
  \gAut(\Delta)\ar[r]&1}
  \]
  The differential $d\Phi$ is injective, since the restriction of
  $\Phi$ to the connected component of the identity
  $\gPGO^+(A,\sigma,f)$ is the identity map. Moreover,
  Lemma~\ref{PGOhyp.lem} shows that over an algebraic closure $F_\alg$
  the map
  \[
  \Phi_\alg\colon\gPGO(A,\sigma,f)(F_\alg) \to
  \gAut(\gPGO^+(A,\sigma,f))(F_\alg)
  \]
  is injective. It follows by \cite[(22.2)]{KMRT} that $\Phi$ is
  injective, and likewise $\Psi$ is injective. We have
  $\gAut(\Delta)\simeq \gAut_F(Z)$ if $n\neq4$, and
  $\gAut(\Delta)\simeq \gAut_F(F\times Z)$ if $n=4$. Since $Z$ is
  assumed to be a field, in each case the group of $F$-rational points
  is
  \[
  \Aut(\Delta)\simeq \Z/2\Z\simeq \Aut_F(Z).
  \]
  Therefore, the diagram above yields the following diagram with exact
  rows:
  \[
  \xymatrix{
  1\ar[r]&\PGO^+(A,\sigma,f) \ar[r] \ar@{=}[d] & \PGO(A,\sigma, f)
  \ar[r] \ar[d]_{\Phi_F}&\Z/2\Z\ar@{=}[d]\\
  1\ar[r]&\PGO^+(A,\sigma,f)\ar[r]&\Aut(\gPGO^+(A,\sigma,f)) \ar[r]&
  \Z/2\Z}
  \]
  It follows that $\Phi_F$ is an isomorphism, which proves~(2) and (3)
  of Proposition~\ref{orthogonal.prop}.
\end{proof}

\begin{remarkk}
  \label{orthogonal.rem}
  (i) If the algebra $A$ is split, which means that
  $\gPGO^+(A,\sigma,f)=\gPGO^+(V,q)$ for some quadratic space $(V,q)$,
  then (Out~3) holds, since each quadratic space admits improper
  isometries of order~$2$.
  
  (ii) The arguments in the proof of
  Proposition~\ref{orthogonal.prop} also apply in the case
  where $Z\simeq F\times F$. It follows that in this case
    (Out~1), (Out~2), and (Out~3) are equivalent, and hold if and only
    if $A$ is split. Thus, adjoint groups of inner type $\sfD_n$
  admit outer automorphisms of order~$2$ whenever the Tits class
  obstruction vanishes, as pointed out by Garibaldi~\cite{Gar}.
\end{remarkk}

In the outer case, condition (Out~3) induces additional restrictions
on the algebra $A$ when its degree is divisible by $4$, as we now
proceed to show:   

\begin{lemma}
  \label{out3impliesnodd}
  Let $G=\gPGO^+(A,\sigma,f)$ for some $F$-algebra with quadratic pair
  $(A,\sigma,f)$, such that $\deg A\equiv 0\mod 4$, so $G$ has
  type $\sfD_n$ with $n$ even.  If $G$ admits an outer automorphism of
  order $2$, then $A$ is split.
\end{lemma}

\begin{proof}
  In view of Remark~\ref{orthogonal.rem}, it suffices to consider the
  case where the center $Z$ of the Clifford algebra $C=C(A,\sigma,f)$
  is a field.  By Proposition~\ref{orthogonal.prop}, if $G$ admits an
  outer automorphism of order $2$, then $(A,\sigma,f)$ admits a
  square-central improper similitude $g$. As explained
  in~\cite[\S\,13.A]{KMRT}, $g$ induces an automorphism $C(g)$ of
  order $2$ of $C$, which commutes with the canonical involution
  $\underline \sigma$. Moreover, since $g$ is improper, $C(g)$ acts
  non trivially on $Z$. Therefore, the fixed points $C^{C(g)}$ form an
  $F$-algebra $C_0$ of the same degree as $C$, and we have
  $C\simeq C_0\otimes_F Z$. Since $C(g)$ commutes with the canonical
  involution $\underline{\sigma}$ of the Clifford algebra,
  $\underline \sigma$ restricts to an $F$-linear involution on $C_0$,
  so $C_0$ has exponent at most $2$. In view of the fundamental
  relations~\cite[(9.12)]{KMRT}, we get that $A$ is Brauer-equivalent
  to $\Cor_{Z/F}(C_0\otimes _F Z)\simeq C_0\otimes C_0\sim 0$, hence
  $A$ is split, as required.
\end{proof}

To prove Theorem~\ref{main.thm}, we will construct in
\S~\ref{Dnexamples.section} below examples of algebras with quadratic
pairs such that either $A$ is split by the discriminant quadratic
extension, yet $(A,\sigma,f)$ does not admit improper similitudes, or
$(A,\sigma,f)$ admits improper similitudes, but no improper
similitudes of order $2$.  We provide examples of degree $2n$ for
arbitrary $n\geq 3$. The index of $A$ is $2$, as required by condition
(Out~1). 

\section{Outer automorphisms and similitudes: the orthogonal case}
\label{orthogonal.section}

Throughout this section, we assume that the base field $F$ has
characteristic different from~$2$. Hence, we consider orthogonal
involutions instead of quadratic pairs. Our goal is to produce
examples of groups of type $^2\sfD_n$, for all $n\geq 3$, for which
(Out~1) holds and (Out~2) fails, or (Out~2) holds and (Out~3)
fails. Before describing the explicit examples, we first recall a few
well-known facts on similitudes of hermitian forms, and we introduce
our main tool in this section, namely ``generic'' sums of hermitian
forms.

By Proposition~\ref{orthogonal.prop}(1), if
$\gPGO^+(A,\sigma)$ satisfies (Out~1), then $A$ is split by the
discriminant quadratic algebra $Z$. In particular, $A$ has index at
most $2$. Moreover, Remark~\ref{orthogonal.rem} shows that we may
assume $A$ is not split. Hence, our main case of interest is when
$A=M_n(Q)$ for some quaternion division algebra $Q$ over $F$. However,
our discussion of generic sums is more general, because we think this
tool could be useful in various other contexts.

\subsection{Similitudes of hermitian forms} 

Let $D$ be a central division $F$-algebra. Assume $D$ carries an
$F$-linear involution $\rho$, and let $\delta=\pm1$. 
Let $(V,h)$ be a $\delta$-hermitian space over $(D,\rho)$. By
definition, an element $g\in \End_DV$ is a similitude of 
$(V,h)$ with multiplier $\mu(g)=\mu$ if
\[
h\bigl(g(x),g(y)\bigr)=\mu\,h(x,y)\qquad\text{for all $x$, $y\in V$.}
\]
We write $\Sim(V,h)$ or $\Sim(h)$ for the group of similitudes of
$(V,h)$, which is also the group of similitudes of $\End_DV$ for the
adjoint involution $\ad_h$. Depending on $\delta$ and the type of the
reference involution $\rho$, this group is a form of an orthogonal or
a symplectic group: 
\[
\Sim(V,h)=
\begin{cases}
  \GO(\End_DV,\ad_h)&\text{if $\ad_h$ is orthogonal},\\
  \GSp(\End_DV,\ad_h)&\text{if $\ad_h$ is symplectic.}
\end{cases}
\]
For the rest of this subsection, let $A=\End_DV$ and $\deg A=2n$, and
suppose $\ad_h$ is orthogonal; this 
case occurs if and only if $\delta=1$ and $\rho$ is orthogonal, or
$\delta=-1$ and $\rho$ is symplectic, see \cite[(4.2)]{KMRT}.
Since $\charac F\neq2$, we may distinguish as follows
between proper and improper similitudes: for
$g\in\Sim(V,h)$, 
taking the reduced norm of each side of the 
equation $\mu(g)=\sigma(g)g$, we see that $\mu(g)^{2n}=\Nrd_A(g)^2$,
hence $\Nrd_A(g)=\pm\mu(g)^{n}$. The similitude $g$ is 
proper if $\Nrd_A(g)=\mu(g)^n$, and improper if
$\Nrd_A(g)=-\mu(g)^n$ (see \cite[(12.24)]{KMRT}).

Suppose now 
$V=V_1\perp\ldots\perp V_r$ for some subspaces $V_1$, \ldots,
$V_r\subset V$, hence $h$ restricts to a nonsingular
$\delta$-hermitian form $h_i$ on each $V_i$. For $i=1$, \ldots, $r$,
let $A_i=\End_DV_i$,
pick $g_i\in A_i$, and let $g=g_1\oplus\cdots\oplus g_r\in A$ be the
map defined by
\[
g(x_1+\cdots+x_r)=g_1(x_1)+\cdots+g_r(x_r) \qquad\text{for $x_1\in
  V_1$, \ldots, $x_r\in V_r$.}
\]

\begin{lemma}
  \label{lem:propsum}
  With the notation above, $g$ is a similitude of $h$ with multiplier
  $\mu$ if and only if each $g_i$ is a similitude of $h_i$ with
  multiplier $\mu$. When this condition holds, the similitude $g$ is
  proper if and only if the number of improper similitudes among
  $g_1$, \ldots, $g_r$ is even.
\end{lemma}

\begin{proof}
  The first part is clear since $h\bigl(g(x),g(y)\bigr)=\mu\, h(x,y)$
  for all $x$, $y\in V$ if and only if $h_i\bigl(g_i(x),g_i(y)\bigr) =
  \mu\, h_i(x,y)$ for all $i$, and all $x$, $y\in V_i$. To prove the
  second part, let 
  $\deg A_i=2n_i$ for $i=1$, \ldots, $r$, hence $n=n_1+\cdots+n_r$,
  and suppose $\Nrd_{A_i}(g_i)=\varepsilon_i\mu^{n_i}$ with
  $\varepsilon_i=\pm1$. We 
  then have
  \[
  \Nrd_A(g)=\prod_{i=1}^r\Nrd_{A_i}(g_i) =
  \bigl(\prod_{i=1}^r\varepsilon_i\bigr) \mu^{n_1+\cdots+n_r}.
  \qedhere
  \]
\end{proof}

We next consider the particular case where $D$ is a quaternion
division algebra $Q$ and $\rho$ is the canonical involution $\ba$,
hence $\delta=-1$. The generalization of Dieudonn\'e's theorem on
multipliers of similitudes \cite[(13.38)]{KMRT} then allows to
distinguish between proper and improper similitudes as follows: a
similitude $g$ of $(V,h)$ is proper if the quaternion algebra
$\bigl(Z,\mu(g)\bigr)_F$ is split (we write simply
$\bigl(Z,\mu(g)\bigr)_F=0$ in this case), and improper if it is
isomorphic to $Q$. For $1$-dimensional skew-hermitian forms, we
have the following more precise result:

\begin{lemma}
  \label{lem:sim1dim}
  Let $q$ be a nonzero pure quaternion in a quaternion division
  algebra $Q$, and let $a=q^2\in F^\times$. Define
  \[
  G_+(a)= \{\mu\in F^\times\mid (a,\mu)_F=0\}
  \quad\text{and}\quad
  G_-(a)=\{\mu\in F^\times\mid (a,\mu)_F= Q\}.
  \]
  Then $G_+(a)$ is the group of multipliers of proper similitudes of
  the skew-hermitian form $\qf{q}$, and $G_-(a)$ is the coset of
  multipliers of improper similitudes of $\qf{q}$. Moreover, the
  improper similitudes of $\qf{q}$ are all square-central.
\end{lemma}

\begin{proof}
  The lemma follows from the explicit description of similitudes of
  $\qf{q}$ given in \cite[(12.18)]{KMRT}: the proper similitudes form
  the multiplicative group $F(q)^\times\subset Q^\times$, while the
  improper similitudes are the elements $u\in Q^\times$ such that
  $uq=-qu$. 
\end{proof}

In the case where each $V_i$ is
$1$-dimensional, Lemma~\ref{lem:propsum} yields:

\begin{lemma}
  \label{diagonalsimilitude}
  Let $q_1$, \dots, $q_n$ be pure quaternions in $Q$, consider the
  skew-hermitian form $h$ over $(Q,\ba)$ defined by
  $h=\qf{q_1,\dots, q_n}$, and let $a_i=q_i^2\in F^\times$. If
  $\mu\in F^\times$ satisfies $(\mu,a_i)_F\in\{0,Q\}$ for all $i$,
  then $(V,h)$ admits a similitude with multiplier $\mu$. Moreover,
  this similitude is proper if and only if the number of pure
  quaternions among $q_1$, \dots, $q_n$ satisfying $(\mu,a_i)_F=Q$ is
  even.
\end{lemma}

\begin{proof}
  From the condition on $\mu$, it follows by Lemma~\ref{lem:sim1dim}
  that each $\qf{q_i}$ admits a similitude $g_i$ with multiplier
  $\mu$. Then $g=g_1\oplus\cdots\oplus g_n$ is a similitude of $h$
  with multiplier $\mu$. Lemma~\ref{lem:propsum} shows that this
  similitude is proper if and only if the number of indices $i$ such
  that $(\mu,a_i)_F=Q$ is even.
\end{proof} 

Of course, most similitudes do not act diagonally, and the multipliers
of similitudes of $(V,h)$ need not satisfy the condition given in the
lemma; nevertheless, as we explain in the next section, this condition
actually characterizes multipliers of similitudes for some particular
involutions, which we call ``generic sums of orthogonal involutions.''

\subsection{Generic sums}
\label{sec:genericsum}

Let $D$ be a central division algebra over an arbitrary field $F$ of
characteristic different from~$2$. Assume $D$ carries an involution
$\rho$ of the first kind, let $\delta=\pm1$, and let $(V_1,h_1)$,
\ldots, $(V_n,h_n)$ be 
$\delta$-hermitian spaces over $(D,\rho)$. 
Consider the field of iterated
Laurent series in $n$ indeterminates
\[
\widehat F = F((t_1))\ldots((t_n)),
\]
and let
\[
\widehat D = D\otimes_F\widehat F 
\quad\text{and}\quad
\widehat V_i=V_i\otimes_F\widehat F
\text{ for $i=1$, \ldots, $n$.}
\]
The involution $\rho$ extends to an involution
$\widehat\rho= \rho\otimes\Id_{\widehat F}$ on $\widehat D$. We also
extend $h_i$ to a $\delta$-hermitian form $\widehat h_i$ on $\widehat
V_i$, and we let
\[
(\widehat V,\widehat h) = 
(\hat V_1\oplus\cdots\oplus \hat V_n, 
\qf{t_1}\hat h_1\perp \ldots \perp \qf{t_n}\hat h_n).
\]
The adjoint involution $\ad_{\hat h}$ is an orthogonal sum, in the
sense of Dejaiffe~\cite{Dej}, of the involutions $\ad_{\hat h_i}$; we
call it a ``generic orthogonal sum'' since each $\hat h_i$ is extended
from an involution $h_i$ defined over $F$, and scaled by some
indeterminate $t_i$.  
We assume throughout that $h_1$, \ldots, $h_n$ are anisotropic, hence
$\hat h$ is anisotropic. Our goal is to relate the multipliers of
similitudes of $(\hat V,\hat h)$ to the multipliers of similitudes of
$(V_1,h_1)$, \ldots, $(V_n,h_n)$, with the help of a norm on the
vector space $\hat V$, i.e., a valuation-like map for which $\hat V$
contains a splitting base (see \cite[\S2]{RTW}). 
More precisely, we prove:
 
\begin{thm}
  \label{gensummainthm}
  Let $(\hat V,\hat h)$ be a ``generic sum'' of $\delta$-hermitian
  spaces 
  $(V_i,h_i)$ for $1\leq i\leq n$, as defined above.
  \begin{enumerate}
  \item If $n\geq3$, every similitude $g\in\Sim(\hat V, \hat h)$ has
    the form $g=\lambda g'$ for some $\lambda\in \hat F^\times$ and
    some similitude $g'$ with multiplier in $F^\times$.
  \item For every similitude $g\in\Sim(\hat V,\hat h)$ such that
    $\mu(g)\in F^\times\subset\hat F^\times$, there exist similitudes
    $g_i\in\Sim(V_i,h_i)$ for $i=1$, \dots, $n$ with
    $\mu(g)=\mu(g_1)=\cdots=\mu(g_n)$.
  \end{enumerate}
\end{thm}

\begin{proof}
The field $\hat F$ carries the $(t_1,\ldots,t_n)$-adic valuation $v$
with value group $\Z^n$ ordered lexicographically from right to
left. This valuation is Henselian; it extends in a unique way to a
valuation on $\hat D$ with value group $\Z^n$. We write again $v$
for this valuation on $\hat D$. Because $\hat h$ is anisotropic and
$v$ is Henselian, we may define a norm $\nu$ on $\hat V$ by the
following formula (see \cite[Prop.~3.1]{Sch}, \cite[Cor. 3.6, Th.~4.6,
Prop.~4.2]{RTW}):
\[
\nu(x) = \half v\bigl(\hat h(x,x)\bigr) \qquad\text{for $x\in \hat
  V$.}
\]
To describe the value set of this norm, let $\varepsilon_i=(0,\ldots,
1,\ldots,0)=v(t_i)$ be the $i$-th element in the standard base of
$\Z^n$. For $\alpha=(\alpha_1,\ldots,\alpha_n)\in\Z^n$ we write
$t^\alpha=t_1^{\alpha_1}\ldots t_n^{\alpha_n}\in \hat F$. Every nonzero vector
$x_i\in\hat V_i$ can be written as a series $x_i=\sum_\alpha
x_{i\alpha}t^\alpha$ with $x_{i\alpha}\in V_i$, where the support
$\{\alpha\mid x_{i\alpha}\neq0\}$ is a well-ordered subset of
$\Z^n$. If $\alpha_0$ is the minimal element in this support, then
since $\hat h(x_{i\alpha_0},x_{i\alpha_0})=t_i \hat h_i(x_{i\alpha_0},
x_{i\alpha_0})\neq0$ we have $v\bigl(\hat  
h(x_i,x_i)\bigr)=2\alpha_0+\varepsilon_i$. Thus,
\[
\nu(\hat
V_i\setminus\{0\})=\half\varepsilon_i+\Z^n\subset(\half\Z)^n.
\]
It follows that $\nu(\hat V_i)\cap \nu(\hat V_j)=\{\infty\}$ for $i\neq
j$. Therefore, for $x=x_1+\cdots+x_n$ with $x_i\in\hat V_i$ for all
$i$, we have
\begin{equation}
\label{eq:min}
\nu(x)=\min\bigl(\nu(x_1),\ldots,\nu(x_n)\bigr).
\end{equation}
Thus, the value set of $\hat V$, for which we use the notation
$\Gamma_{\hat V}$, is
\[
\Gamma_{\hat V}=\{\nu(x)\mid x\in \hat V\setminus\{0\}\} =
\bigcup_{i=1}^n (\half\varepsilon_i+\Z^n)\subset(\half\Z)^n.
\]

We also need to consider the graded structures associated to norms and
valuations. For $\alpha\in\Z^n$ we let
\[
\hat F_{\geq\alpha}=\{a\in F\mid v(a)\geq\alpha\}, \quad
\hat F_{>\alpha}=\{a\in F\mid v(a)>\alpha\},\quad \text{and } \hat
F_\alpha=\hat F_{\geq\alpha}/\hat F_{>\alpha}.
\]
Thus, $\hat F_\alpha$ is a $1$-dimensional vector space over $F$,
spanned by the image of $t^\alpha$. We let
\[
\gr(\hat F)=\bigoplus_{\alpha\in\Z^n} \hat F_\alpha.
\]
For each nonzero $a\in \hat F$, let $\tilde a=a+\hat
F_{>v(a)}\in\gr(\hat F)$. We also let $\tilde0=0\in\gr(\hat F)$, and
note that the multiplication in $\hat F$ induces a multiplication on
$\gr(\hat F)$, which turns this $F$-vector space into a commutative
graded ring in which every nonzero homogeneous element is invertible:
\[
\gr(\hat F)=F[\tilde t_1,\tilde t_1^{-1}, \ldots, \tilde t_n, \tilde
t_n^{-1}].
\]
The same construction can be applied to $\hat D$, yielding the graded
ring $\gr(\hat D)=D\otimes_F\gr(\hat F)$, and also to $\hat V$, yielding
the graded module $\gr(\hat V)$ over $\gr(\hat
D)$. From~\eqref{eq:min} it follows that 
\[
\gr(\hat V)=\gr(\hat V_1)\oplus\cdots\oplus \gr(\hat V_n),
\]
see \cite[Remark~2.6]{RTW}. For each $i$, we have $\gr(\hat
V_i)=V_i\otimes_F\gr(\hat F)$, with a grading shifted by
$\half\varepsilon_i$. Note that the grade sets of $\gr(\hat V_1)$,
\ldots, $\gr(\hat V_n)$, which are the value sets of $\hat V_1$,
\ldots, $\hat V_n$, are pairwise disjoint, hence every homogeneous
component of $\gr(\hat V)$ lies in exactly one $\gr(\hat V_i)$.

Let $\tilde\rho=\rho\otimes\Id_{\gr(F)}$ be the involution of the
first kind on $\gr(\hat D)$ extending $\rho$. By \cite[Th.~4.6,
Prop.~4.2]{RTW}, we have $v(\hat h(x,y))\geq\nu(x)+\nu(y)$ for all $x$,
$y\in\hat V$. Therefore, the $\delta$-hermitian form $\hat h$ induces a
$\delta$-hermitian form $\tilde h$ on $\gr(\hat V)$, defined on homogeneous
elements by
\[
\tilde h(\tilde x,\tilde y) =
\begin{cases}
  \tilde{\hat h(x,y)}&\text{if $v(\hat h(x,y))=\nu(x)+\nu(y)$,}\\
  0&\text{if $v(\hat h(x,y))>\nu(x)+\nu(y)$},
\end{cases}
\]
and extended by bilinearity to $\gr(\hat V)$. Letting $\tilde h_i$
denote the restriction of $\tilde h$ to $\gr(\hat V_i)$, we have
\[
(\gr(\hat V),\tilde h) = (\gr(\hat V_1),\tilde h_1)\perp\ldots\perp
(\gr(\hat V_n),\tilde h_n).
\]
As observed above, we have $\gr(\hat V_i)=V_i\otimes_F\gr(\hat
F)$. The $\delta$-hermitian form $\tilde h_i$ is given by
\[
\tilde h_i(x,y)=h_i(x,y)\otimes\tilde t_i\qquad\text{for $x$, $y\in
  V_i$.}
\]

Now, suppose $g\colon\hat V\to\hat V$ is a similitude of $(\hat V,\hat
h)$, with multiplier $\mu(g)\in\hat F^\times$. For $x\in \hat V$ we
have $\hat h(g(x),g(x))=\mu(g)\hat h(x,x)$, hence
\[
\nu\bigl(g(x)\bigr)=\nu(x)+\half v\bigl(\mu(g)\bigr) \qquad\text{for
  all $x\in\hat V$.}
\]
As a result, $g$ induces a homomorphism of $\gr(\hat D)$-modules
$\tilde g\colon \gr(\hat V)\to\gr(\hat V)$, defined on homogeneous
elements by
\[
\tilde g(\tilde x)=\tilde{g(x)} \qquad\text{for $x\in\hat V$.}
\]
This homomorphism is a similitude of $(\gr(\hat V),\tilde h)$ with
multiplier $\tilde{\mu(g)}$, and it shifts the grading by $\half
v\bigl(\mu(g)\bigr)$. 
It follows that the value set $\Gamma_{\hat V}$, which is the grade
set of $\gr(V)$, is stable under translation by $\half
  v\bigl(\mu(g)\bigr)\in(\half\Z)^n$. We must therefore have for all
  $i=1$, \ldots, $n$
  \[
  \half\varepsilon_i+\half v\bigl(\mu(g)\bigr) \in \bigcup_{\ell=1}^n
  (\half\varepsilon_\ell+\Z^n).
  \]
  Suppose $i$, $j$ are such that $\half\varepsilon_i+\half
  v\bigl(\mu(g)\bigr)\in \half\varepsilon_j+\Z^n$, and $i\neq j$. For
  $k\neq i$, $j$ we then have
  \[
  \half\varepsilon_k+\half v\bigl(\mu(g)\bigr) \in \half\varepsilon_k+
  \half\varepsilon_j - \half\varepsilon_i +\Z^n \not\subset
  \bigcup_{\ell=1}^n(\half\varepsilon_\ell+\Z^n).
  \]
  This contradiction implies that $\half\varepsilon_i+\half
  v\bigl(\mu(g)\bigr) \in\half\varepsilon_i+\Z^n$ for all $i$, hence
  $v\bigl(\mu(g)\bigr)\in 2\Z^n$. Let
$v\bigl(\mu(g)\bigr)=2v(\lambda_0)$ for some $\lambda_0\in \hat 
  F^\times$, hence $v\bigl(\mu(\lambda_0^{-1}g)\bigr)=0$. Consider the
  residue $\overline{\mu(\lambda_0^{-1}g)}=a\in F^\times$. We have
  $\mu(\lambda_0^{-1}g)=a(1+m)$ for some $m\in\hat F$ with
  $v(m)>0$. Since 
  $\hat F$ is Henselian and the characteristic of the residue field
  $F$ is different from~$2$, we may find $\lambda_1\in\hat F$ with
  $\lambda_1^2=1+m$. Then $\mu(\lambda_1^{-1}\lambda_0^{-1}g)=a$, so
  we may write $g=\lambda g'$ with $\lambda=\lambda_0\lambda_1\in
  \hat F^\times$ and $g'=\lambda^{-1}g$. Then $g'\in\Sim(\hat V,\hat
  h)$ and $\mu(g')=a\in F^\times$. The first assertion of the theorem
  is thus proved.

Now, we prove the second assertion. 
Consider a similitude $g\in\Sim(\hat V,\hat h)$ and assume its
multiplier $\mu(g)$ is in $F^\times\subset\hat F^\times$. Since  
$v\bigl(\mu(g)\bigr)=0$, the similitude $\tilde g\in\Sim(\gr(\hat
V),\tilde h)$ preserves the grading. We may therefore consider its 
restriction $g_i$ to the homogeneous component of degree~$\half
\varepsilon_i$, which is $V_i$. Because $\tilde g$ is a similitude
with multiplier $\tilde{\mu(g)}=\mu(g)$ and
\[
\tilde h(x,y) = h_i(x,y)\tilde t_i \qquad\text{for $x$, $y\in V_i$,}
\]
it follows that $g_i$ is a similitude of $(V_i,h_i)$ with multiplier
$\mu(g)$. 
\end{proof}

The last part of the proof above establishes the following result:

\begin{lemma}
  \label{prop:sim2}
  For every similitude $g\in\Sim(\hat V,\hat h)$ such that $\mu(g)\in
  F^\times\subset\hat F^\times$, the similitude 
  $\tilde g\in\Sim(\gr(\hat V),\tilde h)$
  has the form 
  \[
  \tilde g=(g_1\otimes\Id_{\gr(\hat F)}) \oplus\cdots\oplus
  (g_n\otimes\Id_{\gr(\hat F)})
  \] 
  for some similitudes $g_i\in\Sim(V_i,h_i)$ with
  $\mu(g)=\mu(g_1)=\cdots=\mu(g_n)$. 
\end{lemma}

Abusing notation, we write $g_1\oplus\cdots\oplus
g_n$ for $(g_1\otimes\Id_{\gr(\hat F)}) \oplus\cdots\oplus
  (g_n\otimes\Id_{\gr(\hat F)})$.
Note that conversely, given similitudes $g_i\in\Sim(V_i,h_i)$ for
$i=1$, \ldots, $n$ such that $\mu(g_1)=\cdots=\mu(g_n)$, we may define
a similitude $g\in\Sim(\hat V,\hat h)$ such that $\tilde
g=g_1\oplus\cdots\oplus g_n$ and $\mu(g)=\mu(g_1)\in F^\times$ by
\[
g=(g_1\otimes\Id_{\hat F})\oplus\cdots\oplus (g_n\otimes\Id_{\hat F}).
\]

Now, let us apply these results to the setting of a generic orthogonal
sum of $1$-dimensional skew-hermitian forms over a quaternion division
algebra $Q$ over $F$. The following proposition is a key tool for the
examples we produce below.

\begin{prop}
  \label{key.prop}
  Let $Q$ be a quaternion division algebra over $F$, and consider pure
  quaternions $q_1$, \ldots, $q_n$, with respective squares
  $a_1$, \dots, $a_n\in F^\times$. Let $\hat F$ be the field of
  iterated Laurent 
  series in $n$ indeterminates $t_1$, \ldots, $t_n$ over $F$, let
  $\hat Q=Q\otimes_F\hat F$, and consider the involution
  $\sigma$ on $A=M_n(\hat Q)$ adjoint to the skew-hermitian form
  $\hat h=\qf{t_1q_1,\dots, t_nq_n}$. If $n\geq 3$, then
  \begin{enumerate}
  \item The involution $\sigma$ has discriminant
    $\disc\sigma=a_1\ldots a_n\cdot \hat F^{\times2};$
  \item The involution $\sigma$ admits improper similitudes if and
    only if there exist 
    $\varepsilon_1$, \ldots, $\varepsilon_n\in\{\pm1\}$ such that
    $\varepsilon_1\ldots\varepsilon_n=-1$ and
    \[
    G_{\varepsilon_1}(a_1)\cap\ldots\cap
    G_{\varepsilon_n}(a_n)\neq\varnothing.
    \]
  \item The involution $\sigma$ admits square-central improper
    similitudes if and only if $n$ is odd and
    \[
    G_-(a_1)\cap\ldots\cap G_-(a_n)\neq\varnothing.
    \]
  \end{enumerate}
\end{prop}

\begin{proof}
  The discriminant of $\sigma$ is the product of the discriminants of
  the involutions adjoint to $\qf{t_iq_i}$ for all $i$.  Since the
  discriminant of the adjoint involution of $\qf{q}$, for any nonzero
  pure quaternion $q$, is the square class of $q^2$, we get assertion
  (1).

  Suppose that the hermitian form $\hat h$ admits
  improper similitudes. 
  Since it is a generic orthogonal sum, as defined above, of the 
  $1$-dimensional skew-hermitian forms $h_i=\qf{q_i}$, we may apply
  Theorem~\ref{gensummainthm} and Lemma~\ref{prop:sim2}. Therefore,
  since $n\geq 3$, we may find an 
  improper similitude $g$ of $\hat h$ 
  with multiplier $\mu=\mu(g)\in F^\times\subset \hat F^\times$. By
  Lemma~\ref{prop:sim2} we have $\tilde g=g_1\oplus\cdots\oplus
  g_n$ with $g_i\in\Sim(h_i)$ and $\mu(g_i)=\mu$ for $i=1$, \ldots,
  $n$. Because $g$ and $\tilde g$ are improper, the same computation
  as in 
  Lemma~\ref{lem:propsum} shows that the number of improper
  similitudes among $g_1$, \ldots, $g_n$ is odd. Letting
  $\varepsilon_i=+1$ if $g_i$ is proper and $\varepsilon_i=-1$ if
  $g_i$ is improper, we thus have
  \[
  \mu\in G_{\varepsilon_1}(a_1)\cap\ldots\cap G_{\varepsilon_n}(a_n)
  \quad\text{and}\quad \varepsilon_1\ldots\varepsilon_n=-1.
  \]
  
  Assume in addition $g$ is square-central. From $\sigma(g)g=\mu$, we
  get $g^2=\varepsilon \mu$ for some $\varepsilon\in\{\pm 1\}$. Hence
  we also have $\tilde g^2=\varepsilon \tilde \mu=\varepsilon \mu$. By
  Lemma~\ref{prop:sim2}, this occurs if and only if $g_i^2=\varepsilon
  \mu$ for $i=1$, \dots, $n$. 
  Since $g$ is improper, there is at least one $i$ for which $g_i$ is
  improper. From the description of similitudes recalled in the proof
  of Lemma~\ref{lem:sim1dim}, we get that $g_i$ is a pure
  quaternion that anticommutes with $q_i$. Therefore  
  \[
  \mu=\mu(g_i)=\sigma(g_i)g_i=q_i^{-1}\overline{g_i}q_i g_i=g_i^2.
  \]
  It follows that $\varepsilon=1$. 
  Now assume for the sake of contradiction that $g_j$ is proper for
  some $j$. Then $g_j$ is a quaternion that commutes with $q_j$, i.e.,
  $g_j\in F(q_j)$, and it is square-central, hence it belongs to
  $F^\times\cup F^\times q_j$. The first case leads to
  $\mu=\mu(g_j)\in F^{\times 2}$, which is impossible since
  $Q=(\mu,a_i)_F$ is a division algebra. The second case leads to
  $\mu=\mu(g_j)=-g_j^2$, which is impossible since $\varepsilon=1$.  
 Therefore, $g_j$ is improper for all $j$, that is
 $\varepsilon_1=\dots=\varepsilon_n=-1$. Since $g$ is improper, this
 implies $n$ is odd.    
 
  We have thus proved the ``only if'' parts of~(2) and (3).
  The converse statements are easy consequences of
 Lemma~\ref{lem:propsum} and Lemma~\ref{diagonalsimilitude}.  
\end{proof}

\subsection{Examples of groups of type $^2\sfD_n$} 
\label{Dnexamples.section}

With Proposition~\ref{key.prop} in hand, we can now produce explicit
examples of groups of type $\sfD_n$, proving that conditions (Out~1),
(Out~2), and (Out~3) are not equivalent.  

In our examples, the algebra has the form $A=M_n(Q)$ for some 
integer $n\geq 3$, and some quaternion division algebra $Q$ over
$F$. As a preliminary observation concerning condition (Out~1), note
that the set of discriminants of orthogonal involutions on $A$ is
$(-1)^n\nrd_Q(Q^\times)$. This follows easily from the
fact that any quaternion can be written as a product of two pure
quaternions. On the other hand, a quadratic extension $F(\sqrt\delta)$
of $F$ is a splitting field of $Q$ if and only if $Q$ contains a pure
quaternion $q$ such that $\delta=q^2=-\nrd_Q(q)$. Hence, if $n$ is
odd, for any splitting field $F(\sqrt\delta)$, $A$ does admit
orthogonal involutions $\sigma$ with discriminant $\delta$, and
(Out~1) holds for the corresponding group.  As opposed to this, it is
not always true that $A$ admits an involution $\sigma$ for which
(Out~1) holds if $n$ is even, as we now proceed to show.

\subsubsection{Type $^2\sfD_n$ with $n$ even} In this subsection, we
assume $A=M_n(Q)$ with $n=2m$ even, $m\geq2$. We first prove:  

\begin{prop}
Assume $A=M_n(Q)$ with $n$ even. 
The algebra $A$ admits an orthogonal involution $\sigma$ such that $A$
is split by the discriminant quadratic algebra $Z$ of $\sigma$ if and
only if $-1\in \Nrd_Q(Q^\times)$.  
\end{prop}

\begin{proof} 
If $A$ is split by the discriminant algebra $Z=F(\sqrt\delta)$ of some
orthogonal involution $\sigma$, then $\delta=q^2=-\nrd_Q(q)$ for some
pure quaternion $q$, and $\delta=\nrd_A(x)$ for some $\sigma$
skew-symmetric $x\in A$, so that $\delta\in\nrd_Q(Q^\times)$. Hence,
$\delta$ and $-\delta$ are reduced norms, and we get
$-1\in\nrd_Q(Q^\times)$.  

Assume conversely that $-1\in\Nrd_Q(Q^\times)$, and pick an arbitrary
quadratic field $Z=F(\sqrt\delta)$ that splits $Q$. There exists a
pure quaternion $q\in Q^0$ such that $\delta=q^2=-\nrd_Q(q)$. Since
$-1\in\nrd_Q(Q^\times)$, we get $\delta\in\nrd_Q(Q^\times)$, and since
$n$ is even, it follows that there exists an orthogonal involution
$\sigma$ of discriminant $\delta$.  
\end{proof}

In view of Proposition~\ref{orthogonal.prop}, the following result
provides examples of groups 
$\gPGO^+(A,\sigma)$ of type $^2\sfD_n$, with $n$ even and $n\geq 3$,
which admit outer automorphisms but no outer automorphisms of order
$2$.  

\begin{prop}
  \label{expleorth2non3}
  Let $Q$ be a quaternion division algebra such that
  $-1\in\nrd_Q(Q^\times)$, and let $Z$ be a quadratic splitting field
  for $Q$. For every even integer $n\geq 2$ there exists an orthogonal
  involution $\sigma$ of $M_n(Q)$ with discriminant $Z$ such that
  $(A,\sigma)$ admits improper similitudes. Moreover, $(A,\sigma)$
  does not have square-central improper similitudes.
\end{prop}

\begin{proof}
  Since $Z$ is a quadratic splitting field for $Q$, there exists
  $\delta$, $\nu\in F^\times$ such that $Z=F(\sqrt\delta)$ and
  $Q=(\delta,\nu)_F$.  Moreover, since the norm form of $Q$ represents
  $-1$, the quadratic form $\qf{1,-\nu,-\delta, \nu\delta, 1}$ is
  isotropic. After scaling, we get that
  $\qf{-\nu,1,\delta\nu,-\delta, -\nu}$ also is isotropic, hence
  $\qf{1,-\nu}$ and $\qf{\delta,\nu,-\delta\nu}$ represent a common
  value. This means there exists a pure quaternion $q\in Q^0$ such
  that $a=q^2$ is a norm for the quadratic field extension
  $F(\sqrt{\nu})/F$, or equivalently $(a,\nu)_F=0$.  So we have
  $Q=(\delta,\nu)_F=(a\delta,\nu)_F$. Let $q'$ be a pure
  quaternion with square $a\delta$, and let $\sigma$ be the adjoint
  involution with respect to the skew-hermitian form
  $h=\qf{q',q,q,\dots, q}$. Since 
  $n$ is even, $\sigma$ has discriminant $\delta$. Moreover, by
  Lemma~\ref{diagonalsimilitude}, $\sigma$ admits an improper similitude
  with multiplier $\nu$. Since $Q$ is a division
  algebra, the last 
  assertion follows from Lemma~\ref{out3impliesnodd}.
\end{proof}

To produce examples of groups satisfying (Out~1) but with no
outer automorphisms, we use the ``orthogonal generic sums'' defined
above. More precisely, we consider the following:  

\begin{prop}
  \label{noimproperDeven}
  Let $Q$ be a quaternion division algebra. Assume $Q$ contains pure
  quaternions $q_1$, $q_2$, $q_3$ with respective squares $a_1$,
  $a_2$, and $a_3$ such that
  \begin{enumerate}
  \item $Q$ is split by $F(\sqrt{a_1a_2})$;
  \item $Q$ is not split by $F(\sqrt{a_1a_3})$ nor by
    $F(\sqrt{a_2a_3})$.
  \end{enumerate}
  Then the involution $\sigma$ on $A=M_n(\hat Q)$, with $n$ even, $n\geq3$,
  defined as in Proposition~\ref{key.prop} with $q_1$, $q_2$, $q_3$ as
  above and $q_4=\dots=q_n=q_3$, admits no improper similitudes, yet
  $Q$ is split by the discriminant quadratic extension $Z/F$.
\end{prop}

\begin{proof}
  Since $n$ is even, $\sigma$ has discriminant $a_1a_2$, hence the
  first condition guarantees that $Q$ is split by $Z$. It remains to
  prove that $Q$ has no improper similitudes.  By
  Proposition~\ref{key.prop}, this means we have to prove
  \[
  G_{\varepsilon_1}(a_1)\cap\ldots\cap
  G_{\varepsilon_n}(a_n)=\varnothing,
  \]
  for all $\varepsilon_1$, \ldots, $\varepsilon_n\in\{\pm1\}$ such
  that $\varepsilon_1\ldots\varepsilon_n=-1$.  Recall that
  $\mu\in G_+(a_i)$ (respectively $G_{-}(a_i)$) if and only if
  $(\mu,a_i)_F=0$ (respectively $(\mu,a_i)_F=Q$). Since $Q$ is
  a division algebra, it follows that $G_+(a_3)\cap
  G_{-}(a_3)=\varnothing$. Thus, if the
  intersection above is nonempty, then
  $\varepsilon_3=\dots=\varepsilon_n$. Since $n$ is even, we have
  $\varepsilon_3\dots\varepsilon_n=\varepsilon_3^{n-2}=1$. Therefore,
  it is enough to prove that the following intersections are empty:
  \[
  \text{(i)}\quad G_+(a_1)\cap G_-(a_2)\cap
  G_+(a_3)=\varnothing,\qquad \text{(ii)}\quad G_-(a_1)\cap
  G_+(a_2)\cap G_+(a_3)=\varnothing,
  \]
  \[
  \text{(iii)}\quad G_+(a_1)\cap G_-(a_2)\cap
  G_-(a_3)=\varnothing,\qquad \text{(iv)}\quad G_-(a_1)\cap
  G_+(a_2)\cap G_-(a_3)=\varnothing.
  \]
  Assume that some $\mu\in F^\times$ belongs to the intersection (i)
  (respectively (iv)). The two quaternion algebras
  $(\mu,a_1)_F=(\mu,a_3)_F$ are split (respectively equal to $Q$), while
  the third one is $(\mu,a_2)_F=Q$ (respectively is split). In each
  case, we get that $Q=(\mu,a_2a_3)_F$. This is impossible, since we
  assumed that $F(\sqrt{a_2a_3})$ does not split $Q$. Similarly, if
  $\mu$ belongs to the intersection (ii) or (iii), we get
  $Q=(\mu, a_1a_3)_F$, which again is impossible.
\end{proof}

The following example provides an explicit quaternion algebra $Q$
satisfying the conditions of Proposition~\ref{noimproperDeven}, hence
examples of groups $\gPGO^+(A,\sigma)$ of type $^2\sfD_n$ with
$n$ even, $n\geq3$, for which (Out~1) holds but not (Out~2).  

\begin{example}
  \label{ex:2Devenn}
  Consider a field $k$ of characteristic $\neq2$ such that
  $-1\in k^{\times2}$. Assume $k$ is the center of a quaternion
  division algebra $(a_1,a_2)_k$, and let $F=k(r,s,t)$ where $r$, $s$,
  and $t$ are independent indeterminates. Let $Q=(a_1,a_2)_F$ and
  $a_3=a_1r^2+a_2s^2+a_1a_2t^2\in F^\times$. Clearly, $Q$ is a
  quaternion division algebra containing pure quaternions $q_1$,
  $q_2$, $q_3$ with $q_i^2=a_i$ for $i=1$, $2$, $3$. Since
  $-1\in F^{\times2}$, the algebra $Q=(a_1,a_2)_F=(a_1,a_1a_2)_F$ is split by
  $F(\sqrt{a_1a_2})$. 
  If $Q$ is split by $F(\sqrt{a_1a_3})$, then $a_1a_3$ is represented
  over $F$ by the quadratic form $\qf{a_1,a_2,a_1a_2}$, hence (after
  scaling by $a_1$) $a_3$ is represented by $\qf{1,a_2,a_1a_2}$ over
  $F$. Because $r$, $s$, $t$ are indeterminates, Pfister's subform
  theorem \cite[Th.~IX.2.8]{Lam} shows that this condition implies
  that $\qf{a_1,a_2,a_1a_2}\simeq\qf{1,a_2,a_1a_2}$ over $k$, hence
  (by Witt's cancellation theorem or by comparing discriminants)
  $a_1\in k^{\times2}$. This is impossible since $(a_1,a_2)_k$ is a
  division algebra.
  Similarly, if $Q$ is split by $F(\sqrt{a_2a_3})$, then $a_2a_3$ is
  represented by $\qf{a_1,a_2,a_1a_2}$ over $F$, hence $a_3$ is
  represented by $\qf{1,a_1,a_1a_2}$ over $F$, and
  $\qf{a_1,a_2,a_1a_2}\simeq\qf{1,a_1,a_1a_2}$ over $k$, a
  contradiction since $a_2\notin k^{\times2}$. 
  Hence, the quaternion algebra $Q$ satisfies the conditions of
  Proposition~\ref{noimproperDeven}.  
\end{example}

\subsubsection{Type $^2\sfD_n$, with $n$ odd.}
\label{Doddexample.section}
We again use the orthogonal generic sums
defined in \S\ref{sec:genericsum}. More precisely, we have the
following:  

\begin{prop} 
\label{noimproperDodd}
Let $Q$ be a quaternion division algebra. Assume $Q$ contains pure
quaternions $q_1$, $q_2$, $q_3$ with respective squares $a_1$, $a_2$,
and $a_3$ such that  
\begin{enumerate}
\item $Q$ is split by $F(\sqrt{a_1a_2a_3})$; 
\item There is no $\mu\in F^\times$ such that
  $Q=(a_1,\mu)_F=(a_2,\mu)_F=(a_3,\mu)_F$.  
\end{enumerate}
Consider the involution $\sigma$ of $A=M_n(\hat Q)$, with $n$
odd, $n\geq3$, defined as in Proposition~\ref{key.prop}, with  $q_1$,
$q_2$, $q_3$ as above and $q_4=\dots=q_n=q_3$. 
This involution admits no square-central improper similitudes, yet $Q$
is split by the discriminant quadratic extension $Z/F$.  
Moreover, if in addition $-1\notin \nrd_Q(Q^\times)$, then $\sigma$ has no
improper similitudes.  
\end{prop}  

\begin{proof} 
Since $n$ is odd, $\sigma$ has discriminant $a_1a_2a_3$. Therefore
condition (1) guarantees that $Q$ is split by the discriminant
quadratic algebra $Z$.  
Moreover, arguing as in the proof of
Proposition~\ref{noimproperDeven}, and taking into account the fact
that $n$ is now odd, we get that $\sigma$ has improper similitudes if
and only if one of the following intersections is nonempty:  
\[
  \text{(i)}\quad G_+(a_1)\cap G_+(a_2)\cap G_-(a_3),\qquad
  \text{(ii)}\quad G_+(a_1)\cap G_-(a_2)\cap G_+(a_3),
  \]
  \[
  \text{(iii)}\quad G_-(a_1)\cap G_+(a_2)\cap G_+(a_3),\qquad
  \text{(iv)}\quad G_-(a_1)\cap G_-(a_2)\cap G_-(a_3).
  \]
  In addition, we know by Proposition~\ref{key.prop} that $\sigma$ has
  a square-central improper similitude if and only if the fourth
  intersection is nonempty, or equivalently, if there exists $\mu\in
  F^\times$ such that $Q=(\mu,a_i)_F$ for $i=1$, $2$, $3$. This is
  impossible by condition~(2).  
  
If the involution $\sigma$ has an improper similitude, then one of the
intersections (i), (ii) or (iii) is nonempty. So assume for instance
there exists $\mu\in F^\times$ such that $Q=(\mu,a_3)_F$ and
$(\mu,a_1)_F=(\mu,a_2)_F=0$. The first equation shows that there exists a
pure quaternion $z$ such that $\mu=z^2=-\nrd_Q(z)$. On the other hand,
since $(\mu,a_1)_F=0$, there exists a quaternion $z'\in F(q_1)$ such
that $\mu=N_{F(q_1)/F}(z')=\nrd_Q(z')$. Therefore, both $\mu$ and
$-\mu$ are reduced norms, and it follows $-1$ also is a reduced
norm. This concludes the proof of the proposition.   
\end{proof}

Adapting a construction from \cite{ELTW} (see also
\cite[\S10.2.2]{TW}), we now describe an explicit example of a
quaternion algebra satisfying the conditions of
Proposition~\ref{noimproperDodd}, and we use it to give examples of
groups of type $^2\sfD_n$, with $n$ odd, satisfying (Out~1) and not
(Out~2), or (Out~2) and not (Out~3).  

\begin{example}
\label{noimproperlemma:ex}
Let $k$ be an arbitrary field of
characteristic~$0$, and let $F=k(a_1,a_2)$, where $a_1$ and $a_2$ are
independent indeterminates. Consider the quaternion division algebra
$Q=(a_1,a_2)_F$, and let
\begin{equation}
\label{eq:a3}
a_3=a_1\bigl((1-a_1)^2(1+a_2)^2-4(1-a_1)a_2\bigr).
\end{equation}
The algebra $Q$ satisfies the conditions~(1) and~(2) of
Proposition~\ref{noimproperDodd}.  
\end{example}

\begin{proof}
It is clear that $Q$ contains pure quaternions $q_1$, $q_2$ with
$q_1^2=a_1$ and $q_2^2=a_2$.
Computation yields
  \[
  (1-a_1)^2(1+a_2)^2-a_1^{-1}a_3=4(1-a_1)a_2,
  \]
  hence the quaternion algebra $(a_1^{-1}a_3,(1-a_1)a_2)_F$ is
  split. Therefore,
  \[
  (a_3,(1-a_1)a_2)_F\simeq(a_1,(1-a_1)a_2)_F\simeq Q,
  \]
  and it follows that $Q$ contains a pure quaternion $q_3$ with
  $q_3^2=a_3$.

  Another computation yields
  \[
  (1-a_1)^2(1-a_2)^2-a_1^{-1}a_3=4a_1(1-a_1)a_2,
  \]
  hence the quaternion algebra $(a_1^{-1}a_3,a_1(1-a_1)a_2)_F$ is
  split. Since we already observed that $(a_1^{-1}a_3,(1-a_1)a_2)_F$
  is split, it follows that $(a_1^{-1}a_3,a_1)_F$ is split, hence
  \begin{equation}
    \label{eq:ELTW1}
    (a_1,a_3)_F\simeq(a_1,a_1)_F\simeq(a_1,-1)_F.
  \end{equation}
  We thus see that $(a_1,-a_3)_F$ is split, hence
  \[
  Q\simeq(a_1,-a_2a_3)_F\simeq(a_1,a_1a_2a_3)_F.
  \]
  Therefore, $Q$ is split by $F(\sqrt{a_1a_2a_3})$.

  Suppose now that there exists some $\mu\in F^\times$ such that 
  \begin{equation}
    \label{eq:ELTW}
   Q=(a_1,\mu)_F=(a_2,\mu)_F=(a_3,\mu)_F.
    \end{equation} 
    To
  obtain a contradiction, we use valuation theory as in
  \cite[\S10.2.2]{TW}: since $\charac k=0$, we may find on $k$ a
  dyadic valuation $v_0$, with value group some ordered group
  $\Gamma$ and residue field $\overline{k}$ of
  characteristic~$2$. Consider the Gaussian extension $v_1$ of $v_0$
  to $F$, with value group $\Gamma$ and residue field
  $\overline{k}(a_1,a_2)$, and let $v$ be the valuation on $F$
  obtained by composing $v_1$ with the $(1-a_1)$-adic valuation on
  $\overline{k}(a_1,a_2)$. The value group of $v$ is
  $\mathbb{Z}\times\Gamma$ with the right-to-left lexicographic
  ordering, and the residue field is $\overline{k}(a_2)$. It is clear
  that $v$ extends uniquely to $F(\sqrt{a_2})$, and this extension is
  unramified with a purely inseparable residue field extension. In
  \cite[p.~509]{TW}, it is shown that $v$ also extends uniquely to
  $F(\sqrt{a_1})$ and $F(\sqrt{a_1a_3})$, and that these
  extensions are totally ramified.

  Now, since $Q\simeq(a_2,-a_1a_2)_F$ and \eqref{eq:ELTW} holds, we
  see that $-a_1a_2\mu$ is a norm from $F(\sqrt{a_2})$. Because
  $F(\sqrt{a_2})$ is an unramified extension of $F$, it follows that
  $v(-a_1a_2\mu)\in 2v(F^\times)$. Scaling $\mu$ by the square of an
  element in $F^\times$, we may assume $v(-a_1a_2\mu)=0$ and take the
  residue
  $\overline{-a_1a_2\mu}=a_2\overline{\mu}\in\overline{k}(a_2)$. 
  (We can omit the sign, since $\overline k$ has characteristic $2$.)
  Since 
  $Q\simeq(a_1,-a_1a_2)_F$, we also derive from \eqref{eq:ELTW} that
  $-a_1a_2\mu$ is a norm from $F(\sqrt{a_1})$. As $F(\sqrt{a_1})$ is
  totally ramified over $F$, it follows that $\overline{-a_1a_2\mu}\in
  \overline{F}^{\times2}$, hence
  $a_2\overline{\mu}\in\overline{k}^2(a_2^2)$. But \eqref{eq:ELTW}
  also shows that $(a_1a_3,\mu)_F$ is split, hence $\mu$ is a norm
  from the totally ramified extension $F(\sqrt{a_1a_3})$, and
  therefore $\overline{\mu}\in\overline{k}^2(a_2^2)$. We thus reach
  the conclusion that $a_2\in\overline{k}^2(a_2^2)$, a contradiction.
\end{proof}

\begin{cor}
  \label{-1square:cor}
  Let $Q$ be the quaternion algebra of
  Example~\ref{noimproperlemma:ex} and $q_1$, $q_2$, $q_3\in Q$ be
  pure quaternions satisfying $q_i^2=a_i$ for $i=1$, $2$, $3$. Fix an
  odd integer $n\geq3$ and consider as in Proposition~\ref{key.prop}
  \[
  \hat F=F((t_1))\ldots((t_n)),\qquad \hat Q=Q\otimes_F\hat F, \qquad
  A=M_n(\hat Q),
  \]
  and $\sigma$ the involution on $A$ adjoint to the skew-hermitian
  form $\hat h=\qf{t_1q_1,\ldots,t_nq_n}$ with $q_4=\dots=q_n=q_3$. The
  group $\gPGO^+(A,\sigma)$ satisfies (Out~1) but not (Out~3), and it
  satisfies (Out~2) if and only if $-1\in k^{\times2}$.
\end{cor}

Therefore, depending on the base field $k$ we started with, we get the
required examples.  

\begin{proof}
  Proposition~\ref{noimproperDodd}, together with
  Proposition~\ref{orthogonal.prop}, already shows that the group
  $\gPGO^+(A,\sigma)$ satisfies (Out~1) and not (Out~3), and that it
  does not satisfy (Out~2) if $-1\notin\Nrd_Q(Q^\times)$. Therefore,
  it only remains to show that $-1$ is not a reduced norm of $Q$ if
  $-1\notin k^{\times2}$, and that $\sigma$ admits improper
  similitudes if $-1\in k^{\times2}$.

  The first part is clear: the reduced norm of $Q$ is the quadratic
  form
  $$n_Q\simeq\qf{1,-a_1,-a_2,a_1a_2}$$ over $F=k(a_1,a_2)$. Since
  $a_1$ and $a_2$ 
  are indeterminates, this quadratic form represents $-1$ if and only
  if $-1\in k^{\times2}$.

  Now, assume that $-1\in k^{\times2}$. Since
  $Q=(a_1,a_2)_F$, we have $a_1\in 
  G_-(a_2)$. Moreover, because $-1\in k^{\times2}$ the quaternion algebras
  $(a_1,a_1)_F$ and $(a_1,a_3)_F$ are split (see~\eqref{eq:ELTW1}), so
  $a_1\in G_+(a_1)\cap G_+(a_3)$. Therefore,
  \[
  a_1\in G_+(a_1)\cap G_-(a_2)\cap G_+(a_3).
  \]
  Proposition~\ref{key.prop} then shows that $\sigma$ admits improper
  similitudes. 
\end{proof}

\begin{remarkk}
  \label{rem:A3D3}
  As shown in \cite[\S15.D]{KMRT},
  the Clifford algebra construction defines an equivalence of
  categories from the groupoid $\sfD_3(F)$ to the groupoid
  $\sfA_3(F)$. For any central simple algebra $A$ of degree~$6$ with
  orthogonal involution $\sigma$ over a
  field of characteristic different from~$2$, the Clifford algebra
  $C(A,\sigma)$ has degree~$4$ and carries a canonical unitary
  involution $\underline\sigma$, and we have canonical isomorphisms
  (see \cite[(15.26), (15.27)]{KMRT})
  \[
  \gSpin(A,\sigma)\simeq\gSU(C(A,\sigma),\underline\sigma),\qquad
  \gPGO^+(A,\sigma)\simeq\gPGU(C(A,\sigma),\underline\sigma).
  \]
  Therefore, Corollary~\ref{-1square:cor} with $n=3$
  readily yields examples of groups of type $^2\sfA_3$ that
  satisfy~(Out~1) but not (Out~2), or (Out~2) but not (Out~3). 
  In particular, by Proposition~\ref{descent1.prop}, it also provides
  examples of unitary involutions that do not have a descent. In view
  of Theorem~\ref{descent.thm}, we know that the algebra $C(A,\sigma)$
  in these examples is a division algebra of degree $4$.   
\end{remarkk}

For use in \S\ref{sec:outunitex}, we still make a few observations on
the square-central similitudes of the skew-hermitian form of
Corollary~\ref{-1square:cor} in the particular case where $n=3$, i.e.,
\[
\hat h=\qf{t_1q_1,t_2q_2,t_3q_3}
\]
with $q_1$, $q_2$, $q_3$ as in Example~\ref{noimproperlemma:ex}.

\begin{lemma}
  \label{lem:propersimhath}
  Assume $-1\in k^{\times2}$.
  Every square-central similitude $g$ of $\hat h$ is proper and satisfies
  \[
  g^2=\mu(g)\in \hat F^{\times2}.
  \]
\end{lemma}

\begin{proof}
  Let $g^2=\lambda\in\hat F^\times$. We have
  $\lambda^2=\mu(g^2)=\mu(g)^2$, hence $\lambda=\pm\mu(g)$. Scaling
  $g$, we may assume by Theorem~\ref{gensummainthm} that $\mu(g)\in
  F^\times$, hence also $\lambda\in F^\times$. By
  Proposition~\ref{prop:sim2} we then have
  \[
  \tilde g=g_1\oplus g_2\oplus g_3
  \]
  for some $g_i\in\Sim(h_i)$ with $\mu(g)=\mu(g_1)=\mu(g_2)=\mu(g_3)$
  and $\lambda=\tilde g^2=g_1^2=g_2^2=g_3^2$. By
  Example~\ref{noimproperlemma:ex} and
  Proposition~\ref{noimproperDodd} the similitude $g$ 
  must be proper since it is square-central. Therefore, the
  number of improper similitudes among $g_1$, $g_2$, $g_3$ is even, so
  at least one of $g_1$, $g_2$, $g_3$ is a proper similitude. If $g_i$
  is proper, then $g_i\in F(q_i)^\times$. Since $g_i$ is
  square-central, it follows that $g_i\in F^\times\cup q_i F^\times$,
  hence $g_i^2\in F^{\times2}\cup a_i F^{\times2}$ and
  $\mu(g_i)=\Nrd_Q(g_i)\in F^{\times2}\cup(-a_i) F^{\times2}$. If
  $g_i$ is improper, then $\bigl(a_i,\mu(g_i)\bigr)_F\simeq Q$: see
  Lemma~\ref{lem:sim1dim}. We now consider the various possibilities:
  \smallbreak\par\noindent
  (1) If $g_1$ is proper and $g_2$, $g_3$ are improper: then
  $\mu(g)\in F^{\times2}\cup(-a_1)F^{\times2}$ and
  $\bigl(a_2,\mu(g)\bigr)_F\simeq \bigl(a_3,\mu(g)\bigr)_F\simeq
  Q$. Since $-1\in k^{\times2}$, the quaternion algebra $(a_3,-a_1)_F$
  is split (see \eqref{eq:ELTW1}) whereas $Q$ is not split, so this
  case is impossible.
  \smallbreak\par\noindent
  (2) If $g_2$ is proper and $g_1$, $g_3$ are improper: then
  $\mu(g)\in F^{\times2}\cup(-a_2)F^{\times2}$ and
  $\bigl(a_1,\mu(g)\bigr)_F\simeq \bigl(a_3,\mu(g)\bigr)_F\simeq
  Q$. Since $Q$ is not split, we must have $\mu(g)\in
  (-a_2)F^{\times2}=a_2F^{\times 2}$, and we get
  $(a_1,a_2)_F=(a_3,a_2)_F$, hence $(a_1a_3,a_2)_F$ is split. By 
  definition of $a_3$ (see \eqref{eq:a3}), this means that the
  quaternion algebra
  \[
  \bigl((1-a_1)((1-a_1)(1+a_2)^2-4a_2),\,a_2\bigr)_F
  \]
  is split. This is a contradiction, since this quaternion algebra is
  ramified for the $(1-a_1)$-adic valuation.
  \smallbreak\par\noindent
  (3) If $g_3$ is proper and $g_1$, $g_2$ are improper: this case is
  excluded just like the previous two, because the quaternion algebra
  $(a_1,a_3)_F$ is split.
  \smallbreak\par\noindent
  The only remaining case is when $g_1$, $g_2$, and $g_3$ are proper,
  hence $\mu(g_i)\in F^{\times2}\cup(-a_i)F^{\times2}$ for each
  $i$. Since $a_1$, $a_2$, and $a_3$ are in different square classes
  and $\mu(g_1)=\mu(g_2)=\mu(g_3)$, it follows that $\mu(g_i)\in
  F^{\times2}$, hence $g_i\in F^\times$ for all $i$. Then
  $\lambda=g_i^2=\mu(g_i)$ for all $i$, hence $g^2=\mu(g)\in F^{\times2}$.
\end{proof}

\section{Outer automorphisms and similitudes: the unitary case}
\label{sec:outunit}

We now turn to the results concerning unitary groups. We already gave
in Remark~\ref{rem:A3D3} examples of groups of type $^2\sfA_3$
satisfying (Out~1) but not (Out~2), or satisfying (Out~2) but not
(Out~3). The other examples we will provide are of the form
$\gPGU(B,\tau)$ with $B$ of index~$2$.
Unitary involutions on algebras of index~$2$ are examined
in detail in \S\ref{subsec:unitsim}, and the 
examples are given in \S\ref{sec:outunitex}. They are based on a generic
construction of hermitian forms of unitary type which is discussed for
division algebras of arbitrary index in \S\ref{sec:genunit}.

The characteristic is arbitrary in
\S\ref{subsec:unitsim}; it is assumed to be different from~$2$ in
\S\ref{sec:genunit} and \S\ref{sec:outunitex}.

\subsection{Similitudes for unitary hermitian forms over a quaternion algebra}
\label{subsec:unitsim}

Let $Q$ be a quaternion division algebra over a field $K$ of arbitrary
characteristic, which is a quadratic separable extension of some
subfield $F$. We write $\iota$ for the nontrivial automorphism of $K$
over $F$.
Let $(B,\tau)$ be an algebra with unitary involution Brauer-equivalent
to $Q$.  We have seen in \S\,\ref{unitary.sec} that outer
automorphisms of $\gPGU(B,\tau)$ are given by $\iota$-semilinear
automorphisms of $(B,\tau)$. In this section, we describe them
explicitly in terms of the underlying hermitian space.

Let $U$ be a
finite-dimensional right $Q$-vector space such that $B=\End_QU$. 
By a theorem of Albert \cite[(2.22)]{KMRT}, unitary involutions on
$B$ exist only if $Q$ has a descent to $F$. We
fix a quaternion $F$-subalgebra $Q_0\subset Q$ and identify
$Q=Q_0\otimes_FK$. Let also $U_0\subset U$ be a $Q_0$-subspace of $U$
such that $U=U_0\otimes_FK$. Thus, $Q_0$ and $U_0$ are the fixed
$F$-algebra and $Q_0$-subspace of the following $\iota$-semilinear
automorphisms of $Q$ and $U$:
\[
\iota_Q=\Id_{Q_0}\otimes\iota,\qquad \iota_U=\Id_{U_0}\otimes\iota.
\]
Similarly, $\End_{Q_0}U_0$ is the $F$-algebra fixed under the
$\iota$-semilinear automorphism of $\End_QU$ that maps $f\in\End_QU$
to the endomorphism $f^\iota$ defined by
\[
f^\iota(x)=\iota_U\bigl(f(\iota_U(x))\bigr)\qquad\text{for all $x\in U$.}
\]
The canonical involution $\ba$ on $Q$ commutes with $\iota_Q$ because
for $x\in Q$
\[
\iota_Q(\overline
x)=\iota_Q(\Trd_Q(x)-x)=\Trd_Q(\iota_Q(x))-\iota_Q(x) =
\overline{\iota_Q(x)}.
\]
Let $\theta=\ba\circ\iota_Q$, a unitary involution on $Q$ which
restricts to the canonical involution on $Q_0$. The
unitary involution $\tau$ on $B=\End_QU$ is the adjoint involution
$\tau=\ad_h$ for some nondegenerate hermitian form 
$h\colon U\times U\to Q$ with respect to $\theta$. 

A conjugate
hermitian form $h^\iota$ is defined on $U$ by
\[
h^\iota(x,y)=\iota_Q\bigl(h(\iota_U(x),\iota_U(y))\bigr)
\quad\text{for $x$, $y\in U$.}
\]
It is readily verified that the adjoint involutions of $h$ and
$h^\iota$ are related as follows:
\begin{equation}
  \label{eq:1}
  \ad_{h^\iota}(f)^\iota = \ad_h(f^\iota)\qquad\text{for all $f\in\End_QU$.}
\end{equation}
We define a map $g\in\End_QU$ to be a \emph{similitude} $(U,h)\to
(U,h^\iota)$ if there exists $\mu\in F^\times$ such that
\[
h^\iota\bigl(g(x),g(y)\bigr)=\mu\,h(x,y)\qquad\text{for all $x$, $y\in
  U$.}
\]
The factor $\mu$ is said to be the \emph{multiplier} of $g$. We write
$\mu(g)$ for the multiplier of $g$, and $\Sim(U,h,h^\iota)$ or
$\Sim(h,h^\iota)$ for the 
set of similitudes $(U,h)\to(U,h^\iota)$. 
\begin{prop}
  \label{prop:outunitex}
  Every $\iota$-semilinear automorphism $\varphi$ of the algebra with unitary involution $(B,\tau)$ has
  the form $\varphi\colon f\mapsto gf^\iota g^{-1}$ for some 
  $g\in \Sim(U,h,h^\iota)$. This
  automorphism $\varphi$ has order~$2$ if and only if
  $g\,g^\iota\in F^\times$.
\end{prop}
\begin{proof}
  It follows from the Skolem--Noether theorem that every
  $\iota$-semilinear automorphism
  $\varphi$ of $\End_QU$ has the form
  $\varphi\colon f\mapsto gf^\iota g^{-1}$ for some $g\in
  \End_QU$. Equation~\eqref{eq:1} shows that $\varphi$ commutes with
  $\ad_h$ if and only if
  $\Int(g)\circ\ad_{h^\iota}=\ad_h\circ\Int(g)$. But
  $\Int(g^{-1})\circ\ad_h\circ\Int(g)$ is the adjoint involution of
  the form $(x,y)\mapsto h(g(x),g(y))$, so $\varphi$ commutes with
  $\tau$ if and only if $g$ is a similitude $(U,h)\to(U,h^\iota)$. The
  last assertion follows by a straightforward computation.
\end{proof}

\begin{remarkk}
  \label{rem:lambdamu}
  For $g\in\Sim(U,h,h^\iota)$ we have $g^\iota\in\Sim(U,h^\iota,h)$
  with $\mu(g^\iota)=\mu(g)$, hence for all $x$, $y\in U$
  \[
  h^\iota\bigl(gg^\iota(x),gg^\iota(y)\bigr) = \mu(g)
  h\bigl(g^\iota(x),g^\iota(y)\bigr) = \mu(g)^2h^\iota(x,y).
  \]
  Therefore, if $gg^\iota=\lambda\in F^\times$, then
  $\lambda^2=\mu(g)^2$, hence $\lambda=\pm\mu(g)$.
\end{remarkk}

Of course, in the discussion above the choice of $Q_0$ is
arbitrary, and $h$ is defined up to a scalar factor. Multiplying $h$
by some nonzero central element $\alpha$ such that
$\iota(\alpha)=-\alpha$, we may assume $h$ is skew-hermitian instead
of hermitian. More generally, for any $q\in Q^\times$ such that
$\theta(q)=-q$, we may consider $\theta'=\Int(q)\circ\theta$ and set
\[
h'(x,y)=q\,h(x,y)\qquad\text{for $x$, $y\in U$.}
\]
Then $h'$ is a nondegenerate skew-hermitian form with respect to
$\theta'$, and clearly $\ad_{h'}=\ad_h$. Let also
$\iota'_Q=\Int(q)\circ\iota_Q$. The condition $\theta_Q(q)=-q$ yields
$\iota_Q(q)=-\overline q$, hence $q\iota_Q(q)\in F^\times$ and
therefore $\iota'_Q$ is a $\iota$-semilinear automorphism of $Q$ of
order~$2$. Letting $Q'_0$ denote the $F$-subalgebra of $Q$ fixed under
$\iota'_Q$, we have
\[
Q=Q'_0\otimes_FK\qquad\text{and}\qquad
\theta'=\ba\circ\iota'_Q=\iota'_Q\circ\ba.
\]

Here is one case where an appropriate choice of $q$ may lead to a
substantial simplification:

\begin{prop}
\label{descent.propbis}
Let $e_1$, \ldots, $e_n$ be an orthogonal $Q$-base of $(U,h)$, and let
\[
h=\qf{q_1,\ldots,q_n}
\] 
be the corresponding digonalization of $h$.
If the
$K$-span of the quaternions $q_1$, \dots, $q_n$ has dimension at most
$3$, then there is a quaternion $q\in Q^\times$ such that the
skew-hermitian form $h'=qh$ over $(Q,\theta')$ has a diagonalization
\[
h'=\qf{qq_1,\ldots,qq_n}
\]
with $qq_i\in Q'_0$ for $i=1$, \ldots, $n$. The skew-hermitian form
$h'$ then restricts to a nondegenerate skew-hermitian form $h'_0$
(over $(Q'_0,\ba)$) on the $Q'_0$-span $U'_0$ of $e_1$, \ldots, $e_n$,
and we have
\[
(B,\tau)=(\End_QU,\ad_h)=(\End_{Q'_0}U'_0,\ad_{h'_0})\otimes_F(K,\iota).
\]
\end{prop}

\begin{proof} 
  Let $S\subset Q$ be the $K$-span of $q_1$, \ldots, $q_n$, and let
  $S^\perp\subset Q$ be the orthogonal of $S$ for the norm form on
  $Q$. Since $\Nrd_Q(q)=q\overline q=\overline q q$ for every $q\in Q$,
  we have
  \[
  S^\perp=\{s\in Q\mid s\overline{q_i}+q_i\overline s=0 \text{ for all
    $i$}\} = \{s\in Q\mid \overline s q_i+\overline{q_i}s=0 \text{ for
    all $i$}\}.
  \]
  The $K$-space $S^\perp$ is stable under $\theta$ because
  $\theta(q_i)=q_i$ for all $i$ and $\theta$ commutes with $\ba$.
  If $\dim S\leq3$, then $\dim
  S^\perp\geq1$, hence we may find $q\in Q^\times$ such that
  $q^{-1}\in S^\perp$  and 
  $\theta(q)=-q$. (Take $q=(s-\theta(s))^{-1}$ for any $s\in
  S^\perp$ such that $\theta(s)\neq s$; if no such $s$ exists we
  must have $S^\perp=\{0\}$ because $\theta$ is $\iota$-semilinear.) Since
  $\theta(q_i)=q_i$ and $\theta(q)=-q$ we have
  $\overline{q_i}=\iota_Q(q_i)$ and $\overline q=-\iota_Q(q)$, hence
  \[
  \iota_Q(q_i)q^{-1}=\overline{q_i}q^{-1}=-{\overline q}^{-1}q_i =
  \iota_Q(q)^{-1}q_i \quad\text{for $i=1$, \ldots, $n$}.
  \]
  Therefore, 
  \[
  \iota'_Q(qq_i)=q\iota_Q(q)\iota(q_i)q^{-1}=qq_i \qquad\text{for
    $i=1$, \ldots, $n$.}
  \]
  We have thus shown $qq_i\in Q'_0$ for $i=1$, \ldots, $n$; the other
  assertions readily follow.
\end{proof}

The condition on the dimension of the $K$-span of $q_1$,
\dots, $q_n$ is automatically satisfied if $n\leq 3$. Therefore,
Theorem~\ref{descent.thm} for $B$ of index~$2$ follows from 
Proposition~\ref{descent.propbis}. The case where $B$ is split was
already considered in Corollary~\ref{GPCor912.cor}. 

Note that the proof does not require any hypothesis on the
characteristic. (Of course, skew-hermitian forms are hermitian in
characteristic~$2$.) 

\subsection{Generic construction of hermitian forms of unitary type}
\label{sec:genunit}

In this section, we fix a central division algebra with involution of
the first kind $(D,\rho)$ over an arbitrary field $F$ of
characteristic different from~$2$. Adjoining to $F$ an indeterminate
$t$, we consider the fields of Laurent series
\[
\hat K=F((t)) \qquad\text{and}\qquad \hat F=F((t^2))\subset \hat K.
\]
We let $\iota$ denote the nontrivial $\hat F$-automorphism of $\hat K$
and
\[
(\hat D,\hat\rho)=(D,\rho)\otimes_F(\hat K,\iota).
\]
Thus, $(\hat D,\hat\rho)$ is a central division algebra over $\hat K$
with unitary involution. Over this division algebra, we construct
hermitian forms of a particular type, as follows: let $(V_1,h_1)$ be a
hermitian space over $(D,\rho)$ and let $(V_2,h_2)$ be a
skew-hermitian space over $(D,\rho)$. Extending scalars, we obtain a
hermitian form $\hat h_1$ on $\hat V_1=V_1\otimes_F\hat K$ and a
skew-hermitian form $\hat h_2$ on $\hat V_2=V_2\otimes_F\hat K$ (over
$(\hat D,\hat\rho)$). We then set
\[
(\hat U,\hat h)=(\hat V_1\oplus\hat V_2,\hat h_1\perp \qf{t}\hat h_2).
\]
Since $\iota(t)=-t$ and $\hat h_2$ is skew-hermitian, the form
$\qf{t}\hat h_2$ is hermitian, hence $\hat h$ is a hermitian form on
$\hat U$ over $(\hat D,\hat\rho)$.
Set
$\hat D_0=D\otimes_F\hat F$; we have $\hat D=\hat D_0\otimes_{\hat F} \hat K$; hence, the algebra $\hat D$ has a descent. 
Define $\iota_{\hat D}=\Id_D\otimes\iota=\Id_{\hat D_0}\otimes \iota$, and $\hat
U_0=(V_1\oplus V_2)\otimes_F\hat F$, $\iota_{\hat U}=\Id_{\hat
  U_0}\otimes\iota$. Every vector $x\in\hat U$ has a unique expression
as a series $x=\sum_i x_i\otimes t^i$ with $x_i\in V_1\oplus V_2$ for
all $i$, and $\iota_{\hat U}(x)=\sum_i x_i\otimes(-t)^i$. The
conjugate hermitian form $\hat h^\iota$ is
\begin{equation}
\label{eq:conjh}
\hat h^\iota=\hat h_1\perp \qf{-t}\hat h_2.
\end{equation}

For the rest of this section, we assume $h_1$ and $h_2$ are
anisotropic, hence $\hat h$ is anisotropic. As in
\S\ref{sec:genericsum}, we use the $t$-adic valuation to obtain
information on the set of similitudes $\Sim(\hat U,\hat h,\hat
h^\iota)$. More precisely, we prove: 

\begin{prop}
  \label{prop:simgenunit}
  Let $(\hat U,\hat h)$ be defined as above by $\hat h=\hat h_1\perp
  \qf{t}\hat h_2$, where $h_1$ (respectively $h_2$) is an anisotropic
  hermitian (respectively skew-hermitian) form over $(D,\rho)$.  
  Every similitude $g\in\Sim(\hat U,\hat h,\hat h^\iota)$ has the form
  $g=\lambda g'$ for some $\lambda\in\hat K^\times$ and some
  similitude $g'\in\Sim(\hat U,\hat h,\hat h^\iota)$ with $\mu(g')\in
  F^\times$. Moreover, on the graded module $\gr(\hat U)$ associated
  to a suitable norm on $\hat U$, the map $g'$ induces a map $\tilde
  g'$ of the form $\tilde g'=g_1\oplus g_2$ for some similitudes
  $g_1\in\Sim(V_1,h_1)$, $g_2\in\Sim(V_2,h_2)$ with
  \[
  \mu(g')=\mu(g_1)=-\mu(g_2).
  \]
\end{prop}

\begin{proof}

Let $v$ be the $t$-adic valuation on $\hat K$. We write again $v$ for
its extension to $\hat D$ and define a $v$-norm on $\hat U$ by
\[
\nu(x)=\half v\bigl(\hat h(x,x)\bigr) \qquad\text{for $x\in\hat U$.}
\]
Thus, we have $\nu(x_1)\in\Z$ for $x_1\in\hat V_1$, $\nu(x_2)\in\half+\Z$ for
$x_2\in\hat V_2$, and
\[
\nu(x_1+x_2)=\min\bigl(\nu(x_1),\nu(x_2)\bigr)\in\half\Z
\quad\text{for $x_1\in\hat V_1$ and $x_2\in\hat V_2$.}
\]
In view of \eqref{eq:conjh} it follows that $v\bigl(\hat h(x,x)\bigr)
=v\bigl(\hat h^\iota(x,x)\bigr)$, hence
\begin{equation}
  \label{eq:genunit1}
  \nu(x)=\half v\bigl(\hat h^\iota(x,x)\bigr) \qquad\text{for $x\in
    \hat U$.}
\end{equation}
The graded module $\gr(\hat U)$ is defined as in
\S\ref{sec:genericsum}. It carries a hermitian form $\tilde h$ and we
have
\[
(\gr(\hat U),\tilde h)=(\gr(\hat V_1),\tilde h_1)\perp (\gr(\hat
V_2),\tilde h_2),\qquad
(\gr(\hat U),\tilde h^\iota)=(\gr(\hat V_1),\tilde h_1)\perp (\gr(\hat
V_2),-\tilde h_2)
\]
where the hermitian forms $\tilde h_1$, $\tilde h_2$ are given by
\[
\tilde h_1(x_1,y_1)=h_1(x_1,y_1) \qquad\text{and}\qquad
\tilde h_2(x_2,y_2) =\tilde t\, h_2(x_2,y_2)
\]
for $x_1$, $y_1\in V_1$ and $x_2$, $y_2\in V_2$. 

Now, suppose $g\colon(\hat U,\hat h)\to (\hat U,\hat h^\iota)$ is a
similitude. From $\hat h^\iota\bigl(g(x),g(x)\bigr)=\mu(g) \hat
h(x,x)$ it follows by \eqref{eq:genunit1} that
\[
\nu\bigl(g(x)\bigr)=\nu(x)+\half v\bigl(\mu(g)\bigr) \qquad\text{for
  $x\in \hat U$.}
\]
Therefore, $g$ induces a similitude $\tilde g\colon(\gr(\hat U),\tilde
h) \to (\gr(\hat U),\tilde h^\iota)$, which shifts the grading by
$\half v\bigl(\mu(g)\bigr)$. Note that $\half
v\bigl(\mu(g)\bigr)\in\Z$ because $\mu(g)\in\hat F\subset\hat K$. Therefore,
$\gr(\hat V_1)$ and $\gr(\hat V_2)$ are invariant under $\tilde g$. If
$\mu(g)\in F^\times$, the restriction of $\tilde g$ to
$V_1\subset\gr(\hat V_1)$ (resp.\ to $V_2\subset\gr(\hat V_2)$) is a
similitude $g_1\in\Sim(V_1,h_1)$ (resp.\ $g_2\in\Sim(V_2,h_2)$), and
we write (with a slight abuse of notation) $\tilde g=g_1\oplus g_2$.

  Since $\mu(g)\in \hat F^\times$ we have $v\bigl(\mu(g)\bigr)\in2\Z$
  hence there exists $\lambda_0\in\hat K^\times$ such that
  $v\bigl(\mu(g)\bigr)=2v(\lambda_0)$. Then
  $v\bigl(\mu(\lambda_0^{-1}g)\bigr)=0$ and we may find $a\in
  F^\times$, $m\in\hat F^\times$ with $\mu(\lambda_0^{-1}g)=a(1+m)$
  and $v(m)>0$. Arguing as in the proof of the first assertion of Theorem~\ref{gensummainthm},
  we find $\lambda_1\in\hat F^\times$ such that $\lambda_1^2=1+m$, and
  set $\lambda=\lambda_0\lambda_1$. Then $g'=\lambda^{-1}g\in\Sim(\hat
  U,\hat h,\hat h^\iota)$ and $\mu(g')=a\in F^\times$. The equation
  \[
  \tilde h^\iota\bigl(g'(x),g'(y)\bigr)=a\, \tilde h(x,y)
  \qquad\text{for $x$, $y\in\gr(\hat U)$}
  \]
  yields in particular
  \[
  h_1\bigl(g'(x_1),g'(y_1)\bigr) = a\, h_1(x_1,y_1) \qquad\text{for
    $x_1$, $y_1\in V_1$}
  \]
  and
  \[
  -\tilde t\, h_2\bigl(g'(x_2),g'(y_2)\bigr)= a\tilde t\, h_2(x_2,y_2)
  \qquad\text{for $x_2$, $y_2\in V_2$.}
  \]
  Therefore, the restriction $g_1$ of $\tilde g$ to $V_1$ is a
  similitude with $\mu(g_1)=a$, and the restriction $g_2$ of $\tilde
  g$ to $V_2$ is a similitude with $\mu(g_2)=-a$.
\end{proof}

\begin{remarkk}
  \label{rem:giota}
  It is readily verified that $\tilde{\iota_{\hat U}(x)}=\tilde x$ for
  all $x\in V_1\oplus V_2$. Therefore, $\tilde g^\iota=\tilde g$ if
  $\mu(g)\in F^\times$.
\end{remarkk}

\subsection{Examples of groups of type $^2\sfA_n$}
\label{sec:outunitex}

In this section, we use Example~\ref{noimproperlemma:ex} together with
the generic construction of \S\ref{sec:genunit} to build examples of
unitary groups for which (Out~1) holds and (Out~2) fails, or (Out~2) holds and (Out~3) fails.

Let $n\geq7$ be an odd integer. Write $n=5+2m$, where $m\geq1$. We
construct groups of type $^2\sfA_n$ as unitary groups of hermitian
forms of dimension~$3+m$ over a quaternion division algebra with unitary
involution. Since the index of the endomorphism algebra is~$2$, these
groups satisfy~(Out~1).

Adjoining independent indeterminates to an arbitrary field $k$ of
characteristic~$0$, we form the field
\[
F=k(a_1,a_2,x_1,\ldots,x_m)((t_1))((t_2))((t_3))
\]
and the quaternion algebra
\[
Q=(a_1,a_2)_F
\]
with its conjugation involution $\ba$. Let $a_3\in F$ be defined by
Equation~\eqref{eq:a3}. Recall from Example~\ref{noimproperlemma:ex}
that $Q$ contains pure quaternions $q_1$, $q_2$, $q_3$ with
$q_i^2=a_i$ for $i=1$, $2$, $3$. Adjoining to $F$ another
indeterminate $t$, form
\[
\hat K=F((t)), \qquad \hat F=F((t^2))\subset \hat K, \qquad \hat Q=
Q\otimes_F\hat K.
\]
Let $\iota$ be the nontrivial $\hat F$-automorphism of $\hat
K$. Consider the unitary involution $\hat\rho=\ba\otimes\iota$ on
$\hat Q$ and the following hermitian form over $(\hat Q,\hat\rho)$:
\[
\hat h=\qf{x_1,\ldots,x_m}\perp \qf{t}\qf{t_1q_1,t_2q_2,t_3q_3}.
\]
Let $\tau=\ad_{\hat h}$ be its adjoint involution on $B=M_{m+3}(\hat
Q)$.

\begin{prop}
  \label{exunit:prop}
  The algebra with involution $(B,\tau)$ does not admit any
  $\iota$-semilinear automorphism of order~$2$. It admits
  $\iota$-semilinear automorphisms if and only if $-1\in k^{\times2}$.
\end{prop}

In view of Proposition~\ref{unitary.prop}, this provides a group
$\gPGU(B,\tau)$ which does not satisfy (Out~3), and satisfies
(Out~2) if and only if $-1\in k^{\times2}$.

\begin{proof}
  Proposition~\ref{prop:outunitex} translates the conditions on
  semilinear automorphisms of $(B,\tau)$ into conditions on
  similitudes of $\hat h$. Thus, we have to show that there are no
  similitudes $g\in\Sim(\hat h,\hat h^\iota)$ such that $gg^\iota\in
  \hat F$, and that $\Sim(\hat h,\hat h^\iota)$ is nonempty if and
  only if $-1\in k^{\times2}$.

  Note that the form $\hat h$ is obtained by the generic construction
  of \S\ref{sec:genunit}, with $(D,\rho)=(Q,\ba)$ and
  \[
  h_1=\qf{x_1,\ldots,x_m},\qquad h_2=\qf{t_1q_1,t_2q_2,t_3q_3}.
  \]
  
  Suppose first $-1\notin k^{\times2}$ and $g\in\Sim(\hat h, \hat
  h^\iota)$. By Proposition~\ref{prop:simgenunit} we may assume
  $\mu(g)\in F^\times$, hence $\tilde g=g_1\oplus g_2$ for some
  similitudes $g_1\in\Sim(h_1)$, $g_2\in\Sim(h_2)$ with
  $\mu(g_1)=-\mu(g_2)$. Since by Corollary~\ref{-1square:cor} $h_2$
  does not admit improper 
  similitudes, the similitude $g_2$ must be proper, hence by
  \cite[(13.38)]{KMRT} $\mu(g_2)$ 
  is a norm from the discriminant extension, which is
  $F(\sqrt{a_1a_2a_3})$. As this extension splits $Q$, it follows that
  $\mu(g_2)$ is a reduced norm of $Q$, hence
  \begin{equation}
    \label{eq:muNrd}
    \qf{\mu(g_2)}\qf{1,-a_1,-a_2,a_1a_2} \simeq \qf{1,-a_1,-a_2,a_1a_2}.
  \end{equation}
  On the other hand, since $g_1$ is a similitude of $h_1$ with
  multiplier $-\mu(g_2)$, we have
  \[
  \qf{-\mu(g_2)} h_1\simeq h_1.
  \]
  It follows that $-\mu(g_2)$ is also the multiplier of a similitude
  of the ``trace'' quadratic form $\varphi(x)=h_1(x,x)$, which is
  \[
  \varphi\simeq \qf{1,-a_1,-a_2,a_1a_2} \qf{x_1,\ldots, x_m}.
  \]
  Taking into account \eqref{eq:muNrd}, we see that
  \[
  \qf{-1} \qf{1,-a_1,-a_2,a_1a_2} \qf{x_1,\ldots, x_m} \simeq
  \qf{1,-a_1,-a_2,a_1a_2} \qf{x_1,\ldots, x_m}.
  \]
  This is impossible because $-1\notin k^{\times2}$ and $a_1$, $a_2$,
  $x_1$, \ldots, $x_m$ are indeterminates. Therefore, $\Sim(\hat
  h,\hat h^\iota)=\varnothing$ if $-1\notin k^{\times2}$.

  Suppose next $-1\in k^{\times2}$. Then $\hat h=\hat h_1 \perp \qf{t}
  \hat h_2$ is clearly isometric to $\hat h^\iota= \hat h_1 \perp
  \qf{-t}\hat h_2$, hence $\Sim(\hat h,\hat h^\iota)$ is not
  empty. Assume $g\in\Sim(\hat h,\hat h^\iota)$ satisfies 
  $gg^\iota=\lambda\in \hat F^\times$. As above, we may scale $g$ and
  assume $\mu(g)\in F^\times$, hence also $\lambda\in F^\times$ since
  $\lambda=\pm\mu(g)$ by Remark~\ref{rem:lambdamu}. By
  Proposition~\ref{prop:simgenunit} we have
  \[
  \tilde g=g_1\oplus g_2
  \]
  for some $g_1\in\Sim(h_1)$, $g_2\in\Sim(h_2)$ with
  $\mu(g)=\mu(g_1)=-\mu(g_2)$. By Remark~\ref{rem:giota}, the
  equation $gg^\iota=\lambda$ yields $\tilde g^2=\lambda$, hence we
  also have $g_1^2=g_2^2=\lambda$. Now, by
  Lemma~\ref{lem:propersimhath} the similitude $g_2$ must be
  proper and satisfy $g_2^2=\mu(g_2)\in
  F^{\times2}$ since it is square-central. Scaling again, we may assume
  \[
  \mu(g)=\mu(g_1)=-\mu(g_2)=-1 \quad\text{and}\quad \tilde
  g^2=g_1^2=g_2^2=1.
  \]
  The following lemma shows that $h_1$ does not have any similitude
  $g_1$ such that $g_1^2=-\mu(g_1)=1$, hence the existence of $g$
  leads to a contradiction and the proof of
  Proposition~\ref{exunit:prop} is complete:
\renewcommand{\qed}{\relax}
\end{proof}

\begin{lemma}
  \label{lem:g1}
  There is no similitude $g\in\Sim(h_1)$ such that $g^2=-\mu(g)=1$.
\end{lemma}

\begin{proof}
  Extending scalars to
  $k(a_1,a_2)((x_1))\ldots((x_m))((t_1))((t_2))((t_3))$, we may regard
  $h_1$ as a generic orthogonal sum of $m$ times the hermitian form
  $\qf{1}$ over the quaternion algebra $H=(a_1,a_2)_{k(a_1,a_2)}$, and
  use the results of \S\ref{sec:genericsum}. If $g\in\Sim(h_1)$ is
  such that $g^2=-\mu(g)=1$, then by Lemma~\ref{prop:sim2} we
  have
  \[
  \tilde g= g_1\oplus\cdots\oplus g_m
  \]
  for some $g_i\in\Sim(\qf{1})=\Sim(H,\ba)$ with
  $g_i^2=-\mu(g_i)=1$. Each $g_i$ is a pure quaternion because
  $g_i^2=-\mu(g_i)$, and $H$ does not contain any pure quaternion with
  square~$1$ because it is not split. We thus obtain a contradiction.
\end{proof}

\section*{Appendix: Trialitarian groups}

Let $G$ be an algebraic group scheme of adjoint type $\sfD_4$ over an
arbitrary field $F$. Via the $*$-action of the absolute Galois group
of $F$ on the Dynkin diagram $\Delta$ of $G$ (see \cite[\S15.5]{Sp})
we may associate to $G$ a cubic \'etale $F$-algebra $L$ such that
\[
\gAut(\Delta)=\gAut_F(L).
\]
If $g$ is the index of the kernel of the Galois action, the type of
$G$ is denoted by $^g\sfD_4$. Thus, if $G$ is of type $^6\sfD_4$, then
$L$ is a noncyclic separable cubic field extension of $F$, so
$\Aut_F(L)=\{\Id\}$ and $G$ does not have any outer automorphism
defined over $F$. If
$G$ is of type $^2\sfD_4$, then $L\simeq F\times Z$ for some separable
quadratic field extension $Z$ of $F$, and $G\simeq\gPGO^+(A,\sigma,f)$
for some quadratic pair $(\sigma,f)$ with discriminant $Z$ over a
central simple $F$-algebra $A$ of degree~$8$: see
\cite[\S17.3.13]{Sp}. This case has been discussed in
\S\ref{orthogonal.section}. For the rest of this appendix, we focus on
types $^1\sfD_4$ and $^3\sfD_4$.

\subsection*{Type $^1\sfD_4$}
In this case $L\simeq F\times F\times F$, hence $\Aut(\Delta)$ is the
symmetric group $\mathfrak{S}_3$, and $G$ may have outer automorphisms
of order~$2$ or $3$.

\begin{prop}
  Let $G$ be an algebraic group scheme of adjoint type $^1\sfD_4$.
  For every nontrivial subgroup $H\subset\Aut(\Delta)$, the following
  conditions are equivalent:
  \begin{enumerate}
  \item[(1)]
  every element in $H$ fixes the Tits class $t_G$;
  \item[(2)]
  $H$ is contained in the image of the canonical map
  $\Aut(G)\to\Aut(\Delta)$;
  \item[(3)]
  there is a subgroup $H'\subset\Aut(G)$ isomorphic to $H$ under the
  canonical map $\Aut(G)\to\Aut(\Delta)$.
  \end{enumerate}
  When $\lvert H\rvert=2$ the conditions above
  hold if and only if $G=\gPGO^+(q)$ for some
  $8$-dimensional quadratic 
  form $q$ with trivial discriminant. When $\lvert H\rvert=3$ or
  $6$, they hold if and only if $G=\gPGO^+(q)$ for some
  $3$-fold quadratic Pfister form $q$.
\end{prop}

Note that the conditions~(1), (2), (3) are analogues of (Out~1),
(Out~2), and (Out~3) respectively.

\begin{proof}
  The implications (3)~$\Rightarrow$~(2)~$\Rightarrow$~(1) are clear,
  hence it suffices to prove (1)~$\Rightarrow$~(3).
  Choose a representation $G\simeq\gPGO^+(A,\sigma,f)$ for some central 
  simple algebra $A$ of degree~$8$ with a quadratic pair $(\sigma,f)$
  of trivial discriminant.
  If $H$ contains an element $\alpha$ of
  order~$2$, we may choose the representation of $G$ in such a way
  that the action of $\alpha$ on the Tits algebras interchanges the
  two components $C_\pm(A,\sigma,f)$ of $C(A,\sigma,f)$,
  see~\cite[(42.3)]{KMRT}. Then (1) 
  implies $C_+(A,\sigma,f)\simeq C_-(A,\sigma,f)$. Similarly, any
  element of order~$3$ in $H$ permutes $A$, $C_+(A,\sigma,f)$, and
  $C_-(A,\sigma,f)$. Thus, in each case we have $C_+(A,\sigma,f)\simeq
  C_-(A,\sigma,f)$ if (1) holds. Using the fundamental relations
  between $A$ and $C(A,\sigma,f)$ in \cite[(9.12)]{KMRT}, we get that
  $A$ is split if (1)
  holds for any nontrivial $H$, and we may then represent $G$ as
  $\gPGO^+(q)$ for some $8$-dimensional quadratic form
  $q$ of trivial discriminant. Since every quadratic space admits
  square-central improper isometries, as pointed out in
  Remark~\ref{orthogonal.rem}, condition (3) holds if $\lvert
  H\rvert=2$. The proof is thus complete in this case.

  If $\lvert H\rvert=3$ or $6$, the preceding arguments show that
  $C_+(A,\sigma,f)$ and $C_-(A,\sigma,f)$ are isomorphic to $A$ when
  (1) holds, hence they are also split; this means that by scaling
  $q$ we may assume $q$ is a $3$-fold Pfister form. Now, for any
  $3$-fold Pfister form $q$ we may choose a para-Cayley algebra with
  norm form $q$, and use the multiplication in the algebra to define
  outer automorphisms of $\gPGO^+(q)$ of order~$3$, see
  \cite[(35.9)]{KMRT}. Using in addition the conjugation in the para-Cayley
  algebra, we may also define a subgroup of $\Aut(G)(F)$ isomorphic to
  $S_3$, see \cite[(35.15)]{KMRT}.
\end{proof}

\subsection*{Type $^3\sfD_4$}
In this case $L$ is a cyclic cubic field extension of $F$, hence
$\Aut_F(L)\simeq\Z/3\Z$. We may then again consider the conditions
(Out~1), (Out~2), and (Out~3), with the following slight modification:
in (Out~3), the outer automorphism has order~$3$ instead of $2$. If
$\charac F\neq2$, the group $G$ can be represented in the form
$G=\gPGO^+(T)$ for some trialitarian algebra\footnote{Trialitarian
  algebras are not defined in characteristic~$2$.} $T$, see
\cite[(44.8)]{KMRT}. The Allen invariant of $G$ is a central simple
$L$-algebra of degree~$8$.

\begin{prop}
  For $G=\gPGO^+(T)$ of type~$^3\sfD_4$,
  conditions~(Out~1) and (Out~2) 
  are equivalent; they hold if and
  only if the Allen invariant of $T$ is split. Condition~(Out~3) holds
  if and only if $T$ is the endomorphism algebra of a cyclic
  composition induced by a symmetric composition over $F$.
\end{prop}

The first assertion is the main Theorem~A in Garibaldi--Petersson
\cite{GP}.  The second assertion is proved in \cite[Theorem~4.3]{KT}.

As a result of this proposition, it is easy to find examples of groups
of type $^3\sfD_4$ for which (Out~1) and (Out~2) hold while (Out~3)
fails: see \cite[Remark~2.1]{KT}.

\end{document}